\documentclass[11pt]{article}%
\usepackage{makeidx}
\usepackage{amsfonts}
\usepackage{amsmath}
\usepackage{amssymb}
\usepackage{graphicx}
\usepackage{palatino}
\usepackage[colorlinks]{hyperref}%
\providecommand{\U}[1]{\protect\rule{.1in}{.1in}}
\providecommand{\U}[1]{\protect\rule{.1in}{.1in}}
\providecommand{\U}[1]{\protect\rule{.1in}{.1in}}
\hypersetup{colorlinks=blue, linkcolor=cyan, citecolor=green,
filecolor=black, urlcolor=blue }
\newtheorem{theorem}{Theorem}

\newtheorem{corollary}[theorem]{Corollary}

\newtheorem{lemma}[theorem]{Lemma}

\newtheorem{proposition}[theorem]{Proposition}
\newtheorem{remark}[theorem]{Remark}

\newenvironment{proof}[1][Proof]{\noindent\textbf{#1.} }{\ \rule{0.5em}{0.5em}}

\renewcommand{\thefootnote}{\fnsymbol{footnote}}
\begin{document}

\title{Continuity of the Feynman--Kac formula for a generalized parabolic equation}
\author{Etienne Pardoux$^{a}$, Aurel R\u{a}\c{s}canu$^{b}$ \medskip}
\maketitle

\begin{abstract}
It is well--known since the work of Pardoux and Peng \cite{pa-pe/92} that
Backward Stochastic Differential Equations provide probabilistic formulae for
the solution of (systems of) second order elliptic and parabolic equations,
thus providing an extension of the Feynman--Kac formula to semilinear PDEs,
see also Pardoux and R\u{a}\c{s}canu \cite{pa-ra/14}. This method was applied
to the class of PDEs with a nonlinear Neumann boundary condition first by
Pardoux and Zhang \cite{pa-zh/98}. However, the proof of continuity of the
extended Feynman--Kac formula with respect to $x$ (resp. to $(t,x)$) is not
correct in that paper.

Here we consider a more general situation, where both the equation and the
boundary condition involve the (possibly multivalued) gradient of a convex
function. We prove the required continuity. The result for the class of
equations studied in \cite{pa-zh/98} is a Corollary of our main results.

\end{abstract}
\date{}

\textbf{AMS Classification subjects: }60H30, 60F10, 35D40.\medskip

\textbf{Keywords or phrases: }Backward stochastic differential equations;
Feyman-Kac formula; Subdifferential ope\-rators\renewcommand{\thefootnote}{\fnsymbol{footnote}}

\footnotetext{$^{a}${\scriptsize Aix-Marseille Universit\'{e}, CNRS, Centrale
Marseille, I2M, UMR 7373, 13453 Marseille, France, e-mail adress:
etienne.pardoux@univ-amu.fr}$\smallskip$
\par
$^{b}${\scriptsize "Octav Mayer" Mathematics Institute of the Romanian Academy
/Faculty of Mathematics,"Alexandru Ioan Cuza" University, Carol I Blvd., no.
8/11, Ia\c{s}i, 700506, Romania, e-mail adress: aurel.rascanu@uaic.ro}}

\section{Introduction\label{s1}}

The 1998 paper of Pardoux and Zhang \cite{pa-zh/98} has initiated the topics
of the probabilistic study of semilinear parabolic and elliptic systems of
second order partial differential equations with nonlinear Neumann boundary
condition. The idea is to prove that an associated Backward Stochastic
Differential Equation allows to define a certain function of $(t,x)$ (or in
the elliptic case of $x$ alone), which is continuous, and is a viscosity
solution of a certain system of parabolic or elliptic PDEs. Several papers,
see \cite{re-xi/06, ri/09, di-ou/10, ra-zh/10, re-ot/10, am-mr/13}, have used
the above results

However, the continuity is not really proved in \cite{pa-zh/98}. It is claimed
that it follows from several estimates given in earlier sections of the paper,
but this is not really fair. In \cite{ma-ra/15} Maticiuc and Rascanu give a
proof of the continuity result under some additional assumption. In
\cite{li-ta/13} the continuity is shown in the case where all coefficients are
Lipschitz continuous. The difficulty is that not only the solution of forward
SDE depends upon its starting point $x$ (resp. $(t,x)$), but also its local
time on the boundary, which regulates the reflection.

In this paper, we will give the proof of continuity for a class of problems
which is more general than the one considered in \cite{pa-zh/98}, and deduce
the continuity statements from that paper as a Corollary.

More precisely, the aim of this paper is to prove the continuity of the
function $\left(  t,x\right)  \mapsto Y_{t}^{t,x}\overset{def}{=}u\left(
t,x\right)  =\left(  u_{1}\left(  t,x\right)  ,\ldots,u_{m}\left(  t,x\right)
\right)  ^{\ast}:\left[  0,T\right]  \times\overline{D}\rightarrow
\mathbb{R}^{m}$, candidate for being the viscosity solution of the following
system of partial differential equations with a generalized nonlinear
Robin boundary condition and involving multivalued subdifferential
operators of some lower semicontinuous convex functions $\varphi
,\psi:\mathbb{R}^{m}\rightarrow]-\infty,+\infty]$
\begin{equation}
\left\{
\begin{array}
[c]{l}%
-\dfrac{\partial u(t,x)}{\partial t}-\mathcal{L}_{t}u\left(  t,x\right)
+{\partial\varphi}\big(u(t,x)\big)\ni F\big(t,x,u(t,x),\left(  \nabla
u(t,x)\right)  ^{\ast}g(t,x)\big),\smallskip\\
\multicolumn{1}{r}{t\in\left(  0,T\right)  ,\;x\in D,\medskip}\\
\dfrac{\partial u(t,x)}{\partial n}+{\partial\psi}\big(u(t,x)\big)\ni
G\big(t,x,u(t,x)\big),\\
\multicolumn{1}{r}{t\in\left(  0,T\right)  ,\;x\in Bd\left(  \overline
{D}\right)  ,\medskip}\\
u(T,x)=\kappa(x),\;\ x\in\overline{{D}},
\end{array}
\right.  \label{ch5-pvi-1}%
\end{equation}
where $\mathcal{L}_{t}v$, with $v\in C^{2}\left(  \mathbb{R}^{d}%
,\mathbb{R}^{m}\right)  ,$ is a column vector with the coordinates $\left(
\mathcal{L}_{t}v\right)  _{i}~,$ $i\in\overline{1,m},$ given by
\begin{equation}%
\begin{array}
[c]{ll}%
\left(  \mathcal{L}_{t}v\right)  _{i}(x) & =\dfrac{1}{2}\mathrm{Tr}%
\big[g(t,x)g^{\ast}(t,x)D^{2}v_{i}(x)\big]+\left\langle f(t,x),\nabla
v_{i}(x)\right\rangle \medskip\\
& =\dfrac{1}{2}%
{\displaystyle\sum\limits_{j,l=1}^{d}}
\left(  gg^{\ast}\right)  _{j,l}(t,x)\dfrac{\partial^{2}v_{i}(x)}{\partial
x_{j}\partial x_{l}}+%
{\displaystyle\sum\limits_{j=1}^{d}}
f_{j}\left(  t,x\right)  \dfrac{\partial v_{i}(x)}{\partial x_{j}}%
\end{array}
\label{ch5-pvi-1a}%
\end{equation}
$\nabla u$ is the matrix $d\times m$ with the columns $\nabla u_{i}=\left(
\tfrac{\partial u_{i}}{\partial x_{1}},\ldots,\tfrac{\partial u_{i}}{\partial
x_{d}}\right)  ^{\ast},$ $i\in\overline{1,m}$, and $D$ is an open connected
bounded subset of $\mathbb{R}^{d}$ of the form%
\begin{equation}%
\begin{array}
[c]{ll}%
\left(  i\right)  \quad & D=\left\{  x\in\mathbb{R}^{d}:\phi\left(  x\right)
<0\right\}  ,\;\text{where\ }\phi\in C_{b}^{3}\left(  \mathbb{R}^{d}\right)
,\medskip\\
\left(  ii\right)  \quad & Bd\left(  \overline{D}\right)  =\left\{
x\in\mathbb{R}^{d}:\phi\left(  x\right)  =0\right\}  \;\;\text{and }\medskip\\
& \quad\quad\quad\left\vert \nabla\phi\left(  x\right)  \right\vert
=1\;\forall~x\in Bd\left(  \overline{D}\right)  .
\end{array}
\label{sp-30-ipffi}%
\end{equation}
The outward normal derivative of $u\left(  t,\cdot\right)  $ at the point
$x\in Bd\left(  \overline{D}\right)  $ is the column vector%
\[
\tfrac{\partial u\left(  t,x\right)  }{\partial n}=\left(  \tfrac{\partial
u_{1}\left(  t,x\right)  }{\partial n},\ldots,\tfrac{\partial u_{m}\left(
t,x\right)  }{\partial n}\right)  ^{\ast}%
\]
given by%
\[
\frac{\partial u_{i}\left(  t,x\right)  }{\partial n}=\sum_{j=1}^{d}%
\frac{\partial\phi\left(  x\right)  }{\partial x_{j}}\frac{\partial
u_{i}\left(  t,x\right)  }{\partial x_{j}}=\left(  \nabla u_{i}(t,x)\right)
^{\ast}~\nabla\phi\left(  x\right)  ,\;i\in\overline{1,m};
\]
hence%
\[
\frac{\partial u\left(  t,x\right)  }{\partial n}=\left(  \nabla u\left(
t,x\right)  \right)  ^{\ast}~\nabla\phi\left(  x\right)  .
\]

\section{Assumptions and formulation of the problem\label{s2}}

Consider the stochastic basis $\left(  \Omega,\mathcal{F},\mathbb{P},\left(
\mathcal{F}_{s}^{t}\right)  _{s\geq0}\right)  ,$ where the filtration is
generated by a $k-$di\-men\-si\-o\-nal Brownian motion $\left(  B_{r}\right)
_{r\geq0}$ as follows: $\mathcal{F}_{s}^{t}=\mathcal{N}$ if $0\leq s\leq t$
and
\[
\mathcal{F}_{s}^{t}=\sigma\left\{  B_{r}-B_{t}:t\leq r\leq s\right\}
\vee\mathcal{N},\quad\text{if }s>t,
\]
where $\mathcal{N}$ is the family of $\mathbb{P}-$negligible subsets of
$\Omega.$

Denote $S_{d}^{p}\left[  0,T\right]  ,$ $p\geq0,$ the space of (equivalence
classes of) progressively measurable continuous stochastic processes
$X:\Omega\times\left[  0,T\right]  \rightarrow\mathbb{R}^{d}$ such that~:%
\[
\mathbb{E~}\sup_{t\in\left[  0,T\right]  }\left\vert X_{t}\right\vert
^{p}<+\infty,\text{ if }p>0.
\]
By $\Lambda_{d}^{p}\left(  0,T\right)  ,$ $p\geq0,$ denote the space of
(equivalent classes of) progressively measurable stochastic processes
$X:\Omega\times\left]  0,T\right[  \rightarrow\mathbb{R}^{d}$ such that%
\[
\int_{0}^{T}\left\vert X_{t}\right\vert ^{2}dt<+\infty,\quad\mathbb{P}%
-a.s.\,\,\omega\in\Omega,\quad\text{if }p=0,
\]
and%
\[
\mathbb{E~}\left(  \int_{0}^{T}\left\vert X_{t}\right\vert ^{2}dt\right)
^{p/2}<+\infty,\quad\text{if }p>0.
\]

Let $f\left(  \cdot,\cdot\right)  :\mathbb{R}_{+}\times\mathbb{R}%
^{d}\rightarrow\mathbb{R}^{d}$ and $g\left(  \cdot,\cdot\right)
:\mathbb{R}_{+}\times\mathbb{R}^{d}\rightarrow\mathbb{R}^{d\times k}$ are
continuous functions and satisfy: there exist $\mu_{f}\in\mathbb{R}$ and
$\ell_{g}>0$ such that for all $u,v\in\mathbb{R}^{d}$%
\begin{equation}%
\begin{array}
[c]{rl}%
\left(  i\right)  \quad & \left\langle u-v,f(t,u)-f(t,v)\right\rangle
\,\leq\mu_{f}|u-v|^{2},\medskip\\
\left(  ii\right)  \quad & |g(t,u)-g(t,v)|\leq\ell_{g}|u-v|.
\end{array}
\label{sp-30-ip}%
\end{equation}

By Theorem 4.54 and Corollary 4.56 from Pardoux \& R\u{a}\c{s}canu
\cite{pa-ra/14} we infer that for any $(t,x)\in\lbrack0,T]\times\overline{D}$
fixed, there exists a unique pair $\left(  X^{t,x},A^{t,x}\right)
:\Omega\times\left[  0,\infty\right[  \rightarrow\mathbb{R}^{d}\times
\mathbb{R}$ of continuous progressively measurable stochastic processes such
that, $\mathbb{P}-a.s.:$%
\begin{equation}
\left\{
\begin{array}
[c]{rl}%
\left(  j\right)  \; & X_{s}^{t,x}\in\overline{D}\;\text{and }X_{s\wedge
t}^{t,x}=x\text{ for all }s\geq0,\medskip\\
\left(  jj\right)  \; & 0=A_{u}^{t,x}\leq A_{s}^{t,x}\leq A_{v}^{t,x}\text{
for all }0\leq u\leq t\leq s\leq v,\medskip\\
\left(  jjj\right)  \; & X_{s}^{t,x}+%
{\displaystyle\int_{t}^{s}}
\nabla\phi\left(  X_{r}^{t,x}\right)  dA_{r}^{t,x}=x+%
{\displaystyle\int_{t}^{s}}
f\left(  r,X_{r}^{t,x}\right)  dr\smallskip\\
& \multicolumn{1}{r}{+%
{\displaystyle\int_{t}^{s}}
g\left(  r,X_{r}^{t,x}\right)  dB_{r},\ \;\forall~s\geq t,\medskip}\\
\left(  jv\right)  \; & A_{s}^{t,x}=%
{\displaystyle\int_{t}^{s}}
\mathbf{1}_{Bd\left(  \overline{D}\right)  }\left(  X_{r}^{t,x}\right)
dA_{r}^{t,x}~,\;~\forall~s\geq t.
\end{array}
\right.  \label{ch5-pvi-2}%
\end{equation}
Moreover by (4.112) from \cite{pa-ra/14}
\[%
\begin{array}
[c]{r}%
A_{s}^{t,x}=%
{\displaystyle\int\nolimits_{t}^{s}}
\mathcal{L}_{r}\phi(X_{r}^{t,x})dr+%
{\displaystyle\int\nolimits_{t}^{s}}
\left\langle \nabla\phi(X_{r}^{t,x}),g(r,X_{r}^{t,x})dB_{r}\right\rangle
\smallskip\\
-\left[  \phi(X_{s}^{t,x})-\phi\left(  x\right)  \right]  ,
\end{array}
\]
with $\mathcal{L}_{r}$ defined by (\ref{ch5-pvi-1a}).

For every $p\geq1,$ by Proposition 4.55 and Corollary 4.56 from
\cite{pa-ra/14},
\begin{equation}%
\begin{array}
[c]{rl}%
\left(  j\right)  ~ & \left(  t,x\right)  \mapsto\left(  X_{\cdot}%
^{t,x},A_{\cdot}^{t,x}\right)  :\left[  0,T\right]  \times\overline
{D}\rightarrow S_{d}^{p}\left[  0,T\right]  \times S_{1}^{p}\left[
0,T\right]  \medskip\\
& \quad\quad\quad\quad\quad\text{ }\quad\quad\quad\text{is a continuous
mapping,}\medskip\\
\left(  jj\right)  ~ & \sup\limits_{\left(  t,x\right)  \in\left[  0,T\right]
\times\overline{D}}~\Big(\sup\limits_{s\in\left[  0,T\right]  }\mathbb{E}%
e^{\lambda A_{s}^{t,x}}\Big)\leq\exp\left(  C+C~\lambda^{2}\right)  ,
\end{array}
\label{chy5-pvi-2a}%
\end{equation}
for some $C>0$ and every $\lambda>0$. Moreover for every pair of continuous
functions $h_{1},h_{2}:\left[  0,T\right]  \times\overline{D}\rightarrow
\mathbb{R}$ the mapping%
\[
\left(  t,x\right)  \mapsto\mathbb{E}{%
{\displaystyle\int_{t}^{T}}
h_{1}(s,X_{s}^{t,x})ds+\mathbb{E}%
{\displaystyle\int_{t}^{T}}
h_{2}(s,X_{s}^{t,x})\,dA_{s}^{t,x}:}\left[  0,T\right]  \times\overline
{D}\rightarrow\mathbb{R}%
\]
is a.s. continuous.

By the Kolmogorov criterion (choosing a proper version)%
\begin{equation}%
\begin{array}
[c]{r}%
\left(  t,x,s\right)  \mapsto\left(  X_{s}^{t,x}\left(  \omega\right)
,A_{s}^{t,x}\left(  \omega\right)  \right)  :\left[  0,T\right]
\times\overline{D}\times\left[  0,T\right]  \rightarrow\mathbb{R}^{d}%
\times\mathbb{R}\medskip\\
\text{ is continuous, }\mathbb{P}-a.s.\;\omega\in\Omega
\end{array}
\label{chy5-pvi-2aa}%
\end{equation}
and consequently if $\left(  t_{n},x_{n}\right)  \rightarrow\left(
t,x\right)  ,$ then (based also on (\ref{ch5-pvi-2}-j), the boundedness of
$\overline{D}$ and (\ref{chy5-pvi-2a}-jj)) we infer that for all $q>0,$ as
$n\rightarrow\infty$,%
\begin{equation}
\left\vert X_{t_{n}}^{t_{n},x_{n}}-X_{t}^{t_{n},x_{n}}\right\vert +\left\vert
A_{t_{n}}^{t_{n},x_{n}}-A_{t}^{t_{n},x_{n}}\right\vert \rightarrow
0,\quad\mathbb{P}-a.s.\text{ and in }L^{q}\left(  \Omega,\mathcal{F}%
,\mathbb{P}\right)  . \label{xnanxa}%
\end{equation}
Moreover for all $q>0:$%
\[
\lim_{\delta\searrow0}\mathbb{E}\left[  \sup\left\{  \left\vert X_{r}%
^{t,x}-X_{s}^{t,x}\right\vert ^{q}+\left\vert A_{r}^{t,x}-A_{s}^{t,x}%
\right\vert ^{q}:r,s\in\left[  0,T\right]  ,\ \left\vert r-s\right\vert
\leq\delta\right\}  \right]  =0
\]

Let $T>0$ be fixed. We now consider $\left(  Y_{r}^{t,x},Z_{r}^{t,x}%
,U_{r}^{t,x},V_{r}^{t,x}\right)  _{r\in\left[  t,T\right]  }$ the
$\mathbb{R}^{m}\times\mathbb{R}^{m\times k}\times\mathbb{R}^{m}\times
\mathbb{R}^{m}$-valued stochastic process solution of the backward stochastic
variational inequality (BSVI):%
\[%
\begin{array}
[c]{r}%
-dY_{s}^{t,x}+\partial\varphi\left(  Y_{s}^{t,x}\right)  ds+\partial
\psi\left(  Y_{s}^{t,x}\right)  dA_{s}^{t,x}\ni F\left(  s,X_{s}^{t,x}%
,Y_{s}^{t,x},Z_{s}^{t,x}\right)  ds\medskip\\
+G\left(  s,X_{s}^{t,x},Y_{s}^{t,x}\right)  dA_{s}^{t,x}-Z_{s}^{t,x}%
dB_{s},\quad s\in\lbrack t,T)\text{,\ }d\mathbb{P}\text{-a.s.,}\medskip\\
\multicolumn{1}{l}{Y_{T}^{t,x}=\kappa\left(  X_{T}^{t,x}\right)  ,}%
\end{array}
\]
that is%
\begin{equation}
\left\{
\begin{array}
[c]{l}%
Y_{s}^{t,x}+%
{\displaystyle\int_{s}^{T}}
\left(  U_{r}^{t,x}dr+V_{r}^{t,x}dA_{r}^{t,x}\right)  =\kappa\left(
X_{T}^{t,x}\right)  +%
{\displaystyle\int_{s}^{T}}
F\left(  r,X_{r}^{t,x},Y_{r}^{t,x},Z_{r}^{t,x}\right)  dr,\medskip\\
\quad+%
{\displaystyle\int_{s}^{T}}
G\left(  r,X_{r}^{t,x},Y_{r}^{t,x}\right)  dA_{r}^{t,x}-%
{\displaystyle\int_{s}^{T}}
Z_{r}^{t,x}dB_{r},\ \forall~s\in\left[  t,T\right]  ,\;d\mathbb{P}%
\text{-a.s.,}\medskip\\%
{\displaystyle\int_{u}^{v}}
\left\langle U_{r}^{t,x},S_{r}-Y_{r}^{t,x}\right\rangle dr+%
{\displaystyle\int_{u}^{v}}
\varphi\left(  Y_{r}^{t,x}\right)  dr\leq%
{\displaystyle\int_{u}^{v}}
\varphi\left(  S_{r}\right)  dr,\quad d\mathbb{P}\text{-a.s. on }\Omega,\text{
}\smallskip\\
\quad\quad\quad\quad\text{for all }u,v\in\left[  t,T\right]  \text{, }u\leq
v,\text{ for all continuous stochastic process }S;\medskip\\%
{\displaystyle\int_{u}^{v}}
\left\langle V_{r}^{t,x},S_{r}-Y_{r}^{t,x}\right\rangle dA_{r}^{t,x}+%
{\displaystyle\int_{u}^{v}}
\psi\left(  Y_{r}^{t,x}\right)  dA_{r}^{t,x}\leq%
{\displaystyle\int_{u}^{v}}
\psi\left(  S_{r}\right)  dA_{r}^{t,x},\quad d\mathbb{P}\text{-a.s. on }%
\Omega,\smallskip\\
\quad\quad\quad\quad\text{for all }u,v\in\left[  t,T\right]  \text{, }u\leq
v,\text{ for all continuous stochastic process }S.
\end{array}
\right.  \label{ch5-pvi-3}%
\end{equation}
where $F:\mathbb{R}_{+}\times\overline{{D}}\times\mathbb{R}^{m}\times
\mathbb{R}^{m\times k}\rightarrow\mathbb{R}^{m}$, $G:\mathbb{R}_{+}%
\times\overline{D}\times\mathbb{R}^{m}\rightarrow\mathbb{R}^{m}$ and
$\kappa:\overline{{D}}\rightarrow\mathbb{R}^{m}$ are continuous. Assume that
there exist $b_{F},b_{G},\ell_{F}>0$ and $\mu_{F},\mu_{G}\in\mathbb{R}$ (which
can depend on $T$) such that $\forall t\in\left[  0,T\right]  $, $\forall
x\in\overline{{D}}$, $y,\tilde{y}\in\mathbb{R}^{m}$, $z,\tilde{z}\in
\mathbb{R}^{m\times k}$:%
\begin{equation}%
\begin{array}
[c]{rl}%
\left(  i\right)  \text{\ \ } & \left\langle y-\tilde{y}%
,F(t,x,y,z)-F(t,x,\tilde{y},z)\right\rangle \leq\mu_{F}|y-\tilde{y}%
|^{2},\vspace*{2mm}\\
\left(  ii\right)  \text{\ \ } & \left\vert F(t,x,y,z)-F(t,x,y,\tilde
{z})\right\vert \leq\ell_{F}|z-\tilde{z}|,\vspace*{2mm}\\
\left(  iii\right)  \text{\ \ } & \left\vert F(t,x,y,0)\right\vert \leq
b_{F}\left(  1+|y|\right)  ,\vspace*{2mm}\\
\left(  iv\right)  \text{\ \ } & \left\langle y-\tilde{y}%
,G(t,x,y)-G(t,x,\tilde{y})\right\rangle \leq\mu_{G}|y-\tilde{y}|^{2}%
,\vspace*{2mm}\\
\left(  v\right)  \text{\ \ } & \left\vert G(t,x,y)\right\vert \leq
b_{G}\left(  1+|y|\right)  .
\end{array}
\label{h3-0}%
\end{equation}
We also assume that
\begin{equation}%
\begin{array}
[c]{ll}%
\left(  i\right)  \text{\ \ } & \varphi,\psi:\mathbb{R}^{m}\rightarrow
(-\infty,+\infty]\text{ \ are proper convex l.s.c. functions }\medskip\\
\left(  ii\right)  \text{\ \ } & \exists u_{0}\in int\left(  Dom\left(
\varphi\right)  \right)  \cap int\left(  Dom\left(  \psi\right)  \right)
\text{ such that}\medskip\\
& \quad\quad\quad\quad\quad\quad\varphi\left(  y\right)  \geq\varphi\left(
u_{0}\right)  \text{ and }\psi\left(  y\right)  \geq\psi\left(  u_{0}\right)
,\ \forall\;y\in\mathbb{R}.
\end{array}
\label{h4-0}%
\end{equation}
where $Dom\left(  \varphi\right)  =\left\{  y\in\mathbb{R}^{m}:\varphi\left(
y\right)  <\infty\right\}  $ and similarly for $Dom\left(  \psi\right)  .$

We also introduce some \textit{compatibility conditions }: \newline there
exists $M>0$ such that%
\begin{equation}%
\begin{array}
[c]{cc}%
\left(  a\right)  \quad & \sup\limits_{x\in\overline{D}}\left\vert
\varphi\left(  \kappa(x)\right)  \right\vert +\sup\limits_{x\in\overline{D}%
}\left\vert \psi\left(  \kappa(x)\right)  \right\vert =M<\infty
\end{array}
\label{ca1}%
\end{equation}
and there exists $c>0$ such that for all\textit{ }$\varepsilon>0$,
$t\in\left[  0,T\right]  $, $x\in\overline{D}$, $y\in\mathbb{R}^{m}$,
$z\in\mathbb{R}^{m\times k},$%
\begin{equation}%
\begin{array}
[c]{ll}%
\left(  b\right)  \quad & \left\langle \nabla\varphi_{\varepsilon}\left(
y\right)  ,\nabla\psi_{\varepsilon}\left(  y\right)  \right\rangle
\geq0,\medskip\\
\left(  d\right)  \quad & \left\langle \nabla\varphi_{\varepsilon}\left(
y\right)  ,G\left(  t,x,y\right)  \right\rangle \leq c\left\vert \nabla
\psi_{\varepsilon}\left(  y\right)  \right\vert \left[  1+\left\vert G\left(
t,x,y\right)  \right\vert \right]  ,\medskip\\
\left(  e\right)  \quad & \left\langle \nabla\psi_{\varepsilon}\left(
y\right)  ,F\left(  t,x,y,z\right)  \right\rangle \leq c\left\vert
\nabla\varphi_{\varepsilon}\left(  y\right)  \right\vert \left[  1+\left\vert
F\left(  t,x,y,z\right)  \right\vert \right]  ,\medskip\\
\left(  f\right)  \quad & -\left\langle \nabla\varphi_{\varepsilon}\left(
y\right)  ,G\left(  t,x,u_{0}\right)  \right\rangle \leq c\left\vert
\nabla\psi_{\varepsilon}\left(  y\right)  \right\vert \left[  1+\left\vert
G\left(  t,x,u_{0}\right)  \right\vert \right]  ,\medskip\\
\left(  g\right)  \quad & -\left\langle \nabla\psi_{\varepsilon}\left(
y\right)  ,F\left(  t,x,u_{0},0\right)  \right\rangle \leq c\left\vert
\nabla\varphi_{\varepsilon}\left(  y\right)  \right\vert \left[  1+\left\vert
F\left(  t,x,u_{0},0\right)  \right\vert \right]
\end{array}
\label{compatib. assumption}%
\end{equation}
where $\nabla\varphi_{\varepsilon}\left(  y\right)  $, $\nabla\psi
_{\varepsilon}\left(  y\right)  $ are the unique solutions $u$ and $v$,
respectively, of equations%
\[
{\partial\varphi}(y-\varepsilon u)\ni u\;\;\;\text{and \ \ \ }{\partial\psi
}(y-\varepsilon v)\ni v.
\]
(the Moreau-Yosida approximations: see the Annex below).

We remark that the compatibility assumptions are satisfied if, for example,

\begin{itemize}
\item[$\left(  a\right)  $] $\varphi=\psi,$ \newline or in the one dimensional
case (i.e. $m=1$)

\item[$\left(  b\right)  $] If $\varphi,\psi:\mathbb{R}\rightarrow
(-\infty,+\infty]$ are the convex indicator functions
\[
\varphi\left(  y\right)  =\left\{
\begin{array}
[c]{rl}%
0, & \text{if\ }y\in\lbrack a,\infty),\smallskip\\
+\infty, & \text{if\ }y\notin\lbrack a,\infty),
\end{array}
\right.  \;\text{and\ }\psi\left(  y\right)  =\left\{
\begin{array}
[c]{rl}%
0, & \text{if\ }y\in(-\infty,b],\smallskip\\
+\infty, & \text{if\ }y\notin(-\infty,b],
\end{array}
\right.
\]
where $-\infty\leq a<b\leq+\infty$, then
\[
\nabla\varphi_{\varepsilon}\left(  y\right)  =\displaystyle\frac{-\left(
a-y\right)  ^{+}}{\varepsilon}\;\text{and\ }\nabla\psi_{\varepsilon}\left(
y\right)  =\displaystyle\frac{\left(  y-b\right)  ^{+}}{\varepsilon}\,.
\]
In this case the compatibility assumptions (\ref{compatib. assumption}) are
satisfied in particular if there exists $u_{0}\in\left(  a,b\right)  $ such
that for all $\left(  t,x\right)  \in\left[  0,T\right]  \times\overline{D}$
and for all $z\in\mathbb{R}^{1\times k}$ :%
\[%
\begin{array}
[c]{l}%
G\left(  t,x,y\right)  \geq0,\;\text{for all }y<a,\medskip\\
F\left(  t,x,y,z\right)  \leq0,\text{ for all }y>b,\medskip\\
G\left(  t,x,u_{0}\right)  \leq0\quad\text{and }\quad F\left(  t,x,u_{0}%
,0\right)  \geq0,
\end{array}
\]

\end{itemize}

Remark that the backward stochastic variational inequality (\ref{ch5-pvi-3})
satisfies the assumptions of Theorem 5.69 from \cite{pa-ra/14} Therefore
(\ref{ch5-pvi-3}) has a unique progressively measurable solution $\left(
Y^{t,x},Z^{t,x},U^{t,x},V^{t,x}\right)  ,$ with $Y^{t,x}$ having continuous
trajectories, such that for all $\lambda\geq0,$ $\left(  t,x\right)
\in\left[  0,T\right]  \times\overline{D},$%
\[
\mathbb{E}\sup\limits_{r\in\left[  t,T\right]  }e^{2\lambda A_{r}^{t,x}%
}\left\vert Y_{r}^{t,x}\right\vert ^{2}+\mathbb{E}\left(
{\displaystyle\int_{t}^{T}}
e^{2\lambda A_{r}^{t,x}}\left\vert Z_{r}^{t,x}\right\vert ^{2}dr\right)
<\infty.
\]

We extend the stochastic processes from (\ref{ch5-pvi-3}) on $\left[
0,t\right]  $ by the deterministic solution of the following backward
"stochastic" variational inequality ($F=0$, $G=0$) (which again has a unique
solution)
\begin{equation}
\left\{
\begin{array}
[c]{l}%
\begin{array}
[c]{l}%
A_{s}^{t,x}=0\text{, }Z_{s}^{t,x}=0\text{, }\forall~s\in\left[  0,t\right]
,\medskip\\
Y_{s}^{t,x}+%
{\displaystyle\int_{s}^{t}}
U_{r}^{t,x}dr+%
{\displaystyle\int_{s}^{t}}
V_{r}^{t,x}dr=Y_{t}^{t,x}\text{, }\forall~s\in\left[  0,t\right]  ,\medskip\\
U_{r}^{t,x}\in\partial\varphi\left(  Y_{r}^{t,x}\right)  \text{ and }%
V_{r}^{t,x}\in\partial\psi\left(  Y_{r}^{t,x}\right)  \quad\text{a.e. on
}\left[  0,t\right]  .
\end{array}
\end{array}
\right.  \label{extend}%
\end{equation}
Now we can write (\ref{ch5-pvi-3}) as follows%
\begin{equation}
\left\{
\begin{array}
[c]{l}%
Y_{s}^{t,x}+%
{\displaystyle\int_{s}^{T}}
\left(  U_{r}^{t,x}dr+V_{r}^{t,x}dA_{r}^{t,x}\right)  =\kappa\left(
X_{T}^{t,x}\right)  +%
{\displaystyle\int_{s}^{T}}
\mathbf{1}_{\left[  t,T\right]  }\left(  r\right)  ~F\left(  r,X_{r}%
^{t,x},Y_{r}^{t,x},Z_{r}^{t,x}\right)  dr\\
\multicolumn{1}{r}{+%
{\displaystyle\int_{s}^{T}}
\mathbf{1}_{\left[  t,T\right]  }\left(  r\right)  G\left(  r,X_{r}%
^{t,x},Y_{r}^{t,x}\right)  dA_{r}^{t,x}-%
{\displaystyle\int_{s}^{T}}
Z_{r}^{t,x}dB_{r},\ \forall~s\in\left[  0,T\right]  ,\quad\medskip}\\%
{\displaystyle\int_{u}^{v}}
U_{r}^{t,x}\left(  S_{r}-Y_{r}^{t,x}\right)  dr+%
{\displaystyle\int_{u}^{v}}
\varphi\left(  Y_{r}^{t,x}\right)  dr\leq%
{\displaystyle\int_{u}^{v}}
\varphi\left(  S_{r}\right)  dr,\quad d\mathbb{P}\text{-a.s on }\Omega,\text{
}\smallskip\\
\quad\quad\text{for all }u,v\in\left[  0,T\right]  \text{, }u\leq v,\text{ for
any }\mathbb{R}^{m}\text{-valued continuous stochastic process }S;\medskip\\%
{\displaystyle\int_{u}^{v}}
V_{r}^{t,x}\left(  S_{r}-Y_{r}^{t,x}\right)  dA_{r}^{t,x}+%
{\displaystyle\int_{u}^{v}}
\psi\left(  Y_{r}^{t,x}\right)  dA_{r}^{t,x}\leq%
{\displaystyle\int_{u}^{v}}
\psi\left(  S_{r}\right)  dA_{r}^{t,x},\quad d\mathbb{P}\text{-a.s on }%
\Omega,\smallskip\\
\quad\quad\text{for any }u,v\in\left[  0,T\right]  \text{, }u\leq v,\text{ for
all }\mathbb{R}^{m}\text{-valued continuous stochastic process }S;\\
.
\end{array}
\right.  \label{ch5-pvi-3a}%
\end{equation}
(since in particular it is plain that $A_{s}^{t,x}=0$, $\forall~s\in\left[
0,t\right]  $).

If we denote
\[
K_{s}^{t,x}=%
{\displaystyle\int_{0}^{s}}
\left(  U_{r}^{t,x}dr+V_{r}^{t,x}dA_{r}^{t,x}\right)  ,\quad\forall
~s\in\left[  0,T\right]  ,
\]
then as measures on $\left[  0,T\right]  $ we have%
\[
dK_{r}^{t,x}=U_{r}^{t,x}dr+V_{r}^{t,x}dA_{r}^{t,x}\in\partial\varphi\left(
Y_{r}^{t,x}\right)  dr+\partial\psi\left(  Y_{r}^{t,x}\right)  dA_{r}%
^{t,x}\text{ }%
\]
and from the monotonicity of the subdifferential operators we have for all
$\left(  t,x\right)  ,\left(  \tau,y\right)  \in\left[  0,T\right]
\times\overline{D},$%
\begin{equation}
\left\langle Y_{r}^{t,x}-Y_{r}^{\tau,y},dK_{r}^{t,x}-dK_{r}^{\tau
,y}\right\rangle \geq0,\;\text{as measure on }\left[  0,T\right]  .
\label{mon_yk}%
\end{equation}

We highlight (see \cite{ma-ra/15b},  or \cite{pa-ra/14} Proposition 5.46) that
for every $p\geq2$ there exists a positive constant $\hat{C}_{p}$
\textit{depending only upon} $p$ such that for all $t\in\left[  0,T\right]  $,
$x\in\overline{D}$, $s\in\left[  t,T\right]  $ and $\lambda\geq\max\left\{
\left(  \mu_{F}+\ell_{F}^{2}\right)  ,\mu_{G}\right\}  $
\begin{equation}%
\begin{split}
\mathbb{E~} &  \sup\limits_{r\in\left[  0,T\right]  }e^{p\lambda(r+A_{r}%
^{t,x})}\left\vert Y_{r}^{t,x}-u_{0}\right\vert ^{p}+\mathbb{E~}\left(
{\displaystyle\int_{0}^{T}}e^{2\lambda(r+A_{r}^{t,x})}\left\vert Z_{r}%
^{t,x}\right\vert ^{2}dr\right)  ^{p/2}\\
&  \ \quad+\mathbb{E~}\left(  {{\displaystyle\int_{0}^{T}}}e^{2\lambda
(r+A_{r}^{t,x})}\left[  \varphi\left(  Y_{r}^{t,x}\right)  -\varphi\left(
u_{0}\right)  \right]  dr\right)  ^{p/2}\\
&  \ \quad+\mathbb{E~}\left(  {{\displaystyle\int_{0}^{T}}}e^{2\lambda
(r+A_{r}^{t,x})}\left[  \psi\left(  Y_{r}^{t,x}\right)  -\psi\left(
u_{0}\right)  \right]  dA_{r}^{t,x}\right)  ^{p/2}\\
&  \leq\hat{C}_{p}~\mathbb{E}\bigg[e^{p\lambda(T+A_{T}^{t,x})}\left\vert
\kappa\left(  X_{T}^{t,x}\right)  -u_{0}\right\vert ^{p}\\
&  \ \quad+\Big({{\displaystyle\int_{0}^{T}}}e^{\lambda(r+A_{r}^{t,x}%
)}\left\vert F\left(  r,X_{r}^{t,x},u_{0},0\right)  \right\vert dr\Big)^{p}\\
&  \ \quad+\Big({{\displaystyle\int_{0}^{T}}}e^{\lambda(r+A_{r}^{t,x}%
)}\left\vert G\left(  r,X_{r}^{t,x},u_{0}\right)  \right\vert dA_{r}%
^{t,x}\Big)^{p}\bigg].
\end{split}
\label{ch5-pvi-4}%
\end{equation}

Since $\left[  0,T\right]  \times\overline{D}$ is bounded, $X_{r}^{t,x}%
\in\overline{D}$ for all $r\in\left[  0,T\right]  $ and the functions
$\kappa,$ $F$ and $G$ are continuous, there exists a constant $C_{1}$
\textit{independent of }$\left(  t,x\right)  $ such that for all $r\in\left[
0,T\right]  $
\begin{equation}
\left\vert \kappa\left(  X_{T}^{t,x}\right)  \right\vert +\left\vert F\left(
r,X_{r}^{t,x},u_{0},0\right)  \right\vert +\left\vert G\left(  r,X_{r}%
^{t,x},u_{0}\right)  \right\vert \leq C_{1},\quad\mathbb{P}-a.s.\; \label{bk}%
\end{equation}
Taking in account the estimate (\ref{chy5-pvi-2a}-jj) we have that for
every$\lambda\geq\left(  \mu_{F}+\ell_{F}^{2}\right)  \vee\mu_{G}$ and $p>0$
there exists a constant $C_{2}$ \textit{independent of }$\left(  t,x\right)  $
such that
\begin{equation}%
\begin{array}
[c]{l}%
\mathbb{E}\sup\limits_{r\in\left[  0,T\right]  }e^{p\lambda(r+A_{r}^{t,x}%
)}\left\vert Y_{r}^{t,x}\right\vert ^{p}+\mathbb{E~}\left(
{\displaystyle\int_{0}^{T}}
e^{2\lambda(r+A_{r}^{t,x})}\left\vert Z_{r}^{t,x}\right\vert ^{2}dr\right)
^{p/2}\\
\quad+\mathbb{E~}\left(  {%
{\displaystyle\int_{0}^{T}}
}e^{2\lambda(r+A_{r}^{t,x})}\varphi\left(  Y_{r}^{t,x}\right)  dr\right)
^{p/2}\\
\quad+\mathbb{E~}\left(  {%
{\displaystyle\int_{0}^{T}}
}e^{2\lambda(r+A_{r}^{t,x})}\left\vert \psi\left(  Y_{r}^{t,x}\right)
\right\vert dA_{r}^{t,x}\right)  ^{p/2}\\
\leq C_{2}%
\end{array}
\label{ch5-pvi-4a}%
\end{equation}
Moreover for another constant $C_{3}$ \textit{independent of }$\left(
t,x\right)  $ we have
\begin{equation}
\mathbb{E}\Big(%
{\displaystyle\int_{0}^{T}}
e^{2\lambda(r+A_{r}^{t,x})}\left\vert U_{r}^{t,x}\right\vert ^{2}%
dr\Big)+\mathbb{E}\Big(%
{\displaystyle\int_{0}^{T}}
e^{2\lambda(r+A_{r}^{t,x})}\left\vert V_{r}^{t,x}\right\vert ^{2}dA_{r}%
^{t,x}\Big)\leq C_{3} \label{ch5-pvi-5}%
\end{equation}
Since $\left\vert G(t,x,y)\right\vert \leq b_{G}\left(  1+|y|\right)  $ and
$\left\vert F(t,x,y,z)\right\vert \leq\ell_{F}|z|+b_{F}\left(  1+|y|\right)
,$ then every $p>0$ there exists a positive constant $C_{4}$
\textit{independent of} $r,s,t,\tau,\theta\in\left[  0,T\right]  $
\textit{and} $x,y,z\in\overline{D}$ such that%
\begin{equation}%
\begin{array}
[c]{l}%
\mathbb{E~}\left(
{\displaystyle\int_{0}^{T}}
e^{2\lambda(r+A_{r}^{t,x})}\left\vert F\left(  r,X_{r}^{,t,x},Y_{r}^{\tau
,y},Z_{r}^{\tau,y}\right)  \right\vert ^{2}dr\right)  ^{p}\\
\quad+\mathbb{E~}\left(
{\displaystyle\int_{0}^{T}}
e^{2\lambda(r+A_{r}^{t,x})}\left\vert G\left(  r,X_{r}^{,t,x},Y_{r}^{\tau
,y}\right)  \right\vert ^{2}dA_{r}^{t,x}\right)  ^{p}\leq C_{4}%
\end{array}
\label{ch5-pvi-5a}%
\end{equation}

It is clear that the inequalities (\ref{ch5-pvi-4a}), (\ref{ch5-pvi-5}) and
(\ref{ch5-pvi-5a}) are satisfied for all $\lambda\geq0.$

We define
\begin{equation}
u(t,x)=Y_{t}^{t,x},\ \ \ (t,x)\in\lbrack0,T]\times\overline{D},
\label{ch5-pvi-6}%
\end{equation}
which is a deterministic quantity since $Y_{t}^{t,x}$ is $\mathcal{F}_{t}%
^{t}\equiv\mathcal{N}$--measurable. In the next section we shall prove that
$\left(  t,x\right)  \mapsto u(t,x):[0,T]\times\overline{D}\rightarrow
\mathbb{R}^{m}$ is a continuous function

We remark that from the Markov property, we have%
\[
u(s,X_{s}^{t,x})=Y_{s}^{t,x}.
\]

\begin{remark}
\label{rm} We note that in the particular case where $\varphi=\psi\equiv0$, we
are in the situation which was studied in \cite{pa-zh/98}.
\end{remark}

\section{Continuity\label{s3}}

We present here the main result of this paper. The proof will rely upon
several Lemmas which will be proved later in this section.

\begin{theorem}
\label{th:main} \label{p1}Under the above assumptions, the mapping $\left(
t,x\right)  \mapsto u\left(  t,x\right)  =Y_{t}^{t,x}:[0,T]\times
\overline{\mathcal{D}}\rightarrow\mathbb{R}^{m}$ is continuous.
\end{theorem}

\begin{proof}
Let $\left(  t_{n},x_{n}\right)  _{n\ge1} ,\left(  t,x\right)  \in\left[
0,T\right]  \times\overline{D}$ be such that $\left(  t_{n},x_{n}\right)
\rightarrow\left(  t,x\right)  $, as $n\to\infty$.

Denote $\Theta_{s}^{n}=\Theta_{s}^{t_{n},x_{n}}$ and $\Theta_{s}=\Theta
^{0}_{s}=\Theta_{s}^{t,x}$ for $\Theta=X,A,Y,Z,U,V,K.$ From (\ref{ch5-pvi-4a})
and the continuity of the trajectories of $Y^{n}$, for all $q>0$, $n\ge0$,
\[
\lim_{\delta\searrow0}\mathbb{E}\left[  \sup\left\{  \left\vert Y_{r}%
^{n}-Y_{s}^{n}\right\vert ^{q}:r,s\in\left[  0,T\right]  ,\ \left\vert
r-s\right\vert \leq\delta\right\}  \right]  =0.
\]

We have
\[
Y_{s}^{n}-Y_{s}=\kappa\left(  X_{T}^{n}\right)  -\kappa\left(  X_{T}\right)
+\int_{s}^{T}d\mathcal{K}_{r}^{n}-\int_{s}^{T}\left(  Z_{r}^{n}-Z_{r}\right)
dB_{r}%
\]
where%
\begin{align*}
d\mathcal{K}_{r}^{n} &  =d\left(  K_{r}-K_{r}^{n}\right)  \\
&  +\Big[\mathbf{1}_{\left[  t_{n},T\right]  }\left(  r\right)  ~F\left(
r,X_{r}^{n},Y_{r}^{n},Z_{r}^{n}\right)  -\mathbf{1}_{\left[  t,T\right]
}\left(  r\right)  ~F\left(  r,X_{r},Y_{r},Z_{r}\right)  \Big]dr\\
&  +\Big[\mathbf{1}_{\left[  t_{n},T\right]  }\left(  r\right)  ~G\left(
r,X_{r}^{n},Y_{r}^{n}\right)  dA_{r}^{n}-\mathbf{1}_{\left[  t,T\right]
}\left(  r\right)  ~G\left(  r,X_{r},Y_{r}\right)  dA_{r}\Big].
\end{align*}
with $dK_{r}^{n}=U_{r}^{n}dr+V_{r}^{n}dA_{r}^{n}\in\partial\varphi\left(
Y_{r}^{n}\right)  dr+$ $\partial\psi\left(  Y_{r}^{n}\right)  dA_{r}^{n}$ and
$dK_{r}=U_{r}dr+V_{r}dA_{r}\in\partial\varphi\left(  Y_{r}\right)  dr+$
$\partial\psi\left(  Y_{r}\right)  dA_{r}$. Remark that by (\ref{mon_yk}) it
holds%
\[
\left\langle Y_{r}^{n}-Y_{r},dK_{r}-dK_{r}^{n}\right\rangle \leq
0,\quad\text{as a signed measure on }\left[  0,T\right]  .
\]
It is easy to verify that:%
\begin{align*}
&  \left\langle Y_{r}^{n}-Y_{r},\mathbf{1}_{\left[  t_{n},T\right]  }\left(
r\right)  ~F\left(  r,X_{r}^{n},Y_{r}^{n},Z_{r}^{n}\right)  -\mathbf{1}%
_{\left[  t,T\right]  }\left(  r\right)  ~F\left(  r,X_{r},Y_{r},Z_{r}\right)
\right\rangle dr\medskip\medskip\\
&  \leq\left\langle Y_{r}^{n}-Y_{r},\mathbf{1}_{\left[  t_{n},T\right]
}\left(  r\right)  \left[  F\left(  r,X_{r}^{n},Y_{r}^{n},Z_{r}^{n}\right)
-F\left(  r,X_{r}^{n},Y_{r}^{n},Z_{r}\right)  \right]  \right\rangle
dr\medskip\\
&  +\left\langle Y_{r}^{n}-Y_{r},\mathbf{1}_{\left[  t_{n},T\right]  }\left(
r\right)  \left[  F\left(  r,X_{r}^{n},Y_{r}^{n},Z_{r}\right)  -F\left(
r,X_{r}^{n},Y_{r},Z_{r}\right)  \right]  \right\rangle dr\medskip\\
&  +\left\langle Y_{r}^{n}-Y_{r},\mathbf{1}_{\left[  t_{n},T\right]  }\left(
r\right)  F\left(  r,X_{r}^{n},Y_{r},Z_{r}\right)  -\mathbf{1}_{\left[
t,T\right]  }\left(  r\right)  F\left(  r,X_{r},Y_{r},Z_{r}\right)
\right\rangle dr\medskip\medskip\\
&  \leq\ell_{F}\left\vert Y_{r}^{n}-Y_{r}\right\vert \left\vert Z_{r}%
^{n}-Z_{r}\right\vert dr+\mu_{F}\left\vert Y_{r}^{n}-Y_{r}\right\vert
^{2}dr\medskip\\
&  +\left\vert Y_{r}^{n}-Y_{r}\right\vert \left\vert \mathbf{1}_{\left[
t_{n},T\right]  }\left(  r\right)  F\left(  r,X_{r}^{n},Y_{r},Z_{r}\right)
-\mathbf{1}_{\left[  t,T\right]  }\left(  r\right)  F\left(  r,X_{r}%
,Y_{r},Z_{r}\right)  \right\vert dr\medskip\medskip\\
&  \leq\left(  \mu_{F}+\ell_{F}^{2}\right)  \left\vert Y_{r}^{n}%
-Y_{r}\right\vert ^{2}dr+\frac{1}{4}\left\vert Z_{r}^{n}-Z_{r}\right\vert
^{2}dr\medskip\\
&  +\left\vert Y_{r}^{n}-Y_{r}\right\vert \left\vert \mathbf{1}_{\left[
t_{n},T\right]  }\left(  r\right)  F\left(  r,X_{r}^{n},Y_{r},Z_{r}\right)
-\mathbf{1}_{\left[  t,T\right]  }\left(  r\right)  F\left(  r,X_{r}%
,Y_{r},Z_{r}\right)  \right\vert dr
\end{align*}
and%
\begin{align*}
&  \left\langle Y_{r}^{n}-Y_{r},\mathbf{1}_{\left[  t_{n},T\right]  }\left(
r\right)  ~G\left(  r,X_{r}^{n},Y_{r}^{n}\right)  dA_{r}^{n}-\mathbf{1}%
_{\left[  t,T\right]  }\left(  r\right)  ~G\left(  r,X_{r},Y_{r}\right)
dA_{r}\right\rangle \medskip\medskip\\
&  \leq\left\langle Y_{r}^{n}-Y_{r}~,\mathbf{1}_{\left[  t_{n},T\right]
}\left(  r\right)  ~\left[  G\left(  r,X_{r}^{n},Y_{r}^{n}\right)  -G\left(
r,X_{r}^{n},Y_{r}\right)  \right]  dA_{r}^{n}\right\rangle \medskip\\
&  +\left\langle Y_{r}^{n}-Y_{r}~,\left[  \mathbf{1}_{\left[  t_{n},T\right]
}\left(  r\right)  G\left(  r,X_{r}^{n},Y_{r}\right)  -\mathbf{1}_{\left[
t,T\right]  }\left(  r\right)  G\left(  r,X_{r},Y_{r}\right)  \right]
dA_{r}^{n}\right\rangle \medskip\\
&  +\left\langle Y_{r}^{n}-Y_{r}~,\mathbf{1}_{\left[  t,T\right]  }\left(
r\right)  G\left(  r,X_{r},Y_{r}\right)  \left(  dA_{r}^{n}-dA_{r}\right)
\right\rangle \medskip\medskip\\
&  \leq\mu_{G}\left\vert Y_{r}^{n}-Y_{r}\right\vert ^{2}dA_{r}^{n}\\
&  +\left\vert Y_{r}^{n}-Y_{r}\right\vert \left\vert \mathbf{1}_{\left[
t_{n},T\right]  }\left(  r\right)  G\left(  r,X_{r}^{n},Y_{r}\right)
-\mathbf{1}_{\left[  t,T\right]  }\left(  r\right)  G\left(  r,X_{r}%
,Y_{r}\right)  \right\vert dA_{r}^{n}\medskip\\
&  +\left\langle Y_{r}^{n}-Y_{r}~,\mathbf{1}_{\left[  t,T\right]  }\left(
r\right)  G\left(  r,X_{r},Y_{r}\right)  \left(  dA_{r}^{n}-dA_{r}\right)
\right\rangle
\end{align*}
Hence for $\lambda\geq\left(  \mu_{F}+\ell_{F}^{2}\right)  \vee\mu_{G}$%
\begin{align*}
\left\langle Y_{r}^{n}-Y_{r}~,d\mathcal{K}_{r}^{n}\right\rangle  &  \leq
\frac{1}{4}\left\vert Z_{r}^{n}-Z_{r}\right\vert ^{2}dr+\left\vert Y_{r}%
^{n}-Y_{r}\right\vert ^{2}\lambda\left(  dr+dA_{r}^{n}\right)  \medskip\\
&  +\left\vert Y_{r}^{n}-Y_{r}\right\vert dL_{r}^{\left(  n\right)  }%
+dR_{r}^{\left(  n\right)  },
\end{align*}
with%
\begin{equation}%
\begin{array}
[c]{r}%
dL_{r}^{\left(  n\right)  }=\left\vert \mathbf{1}_{\left[  t_{n},T\right]
}\left(  r\right)  F\left(  r,X_{r}^{n},Y_{r},Z_{r}\right)  -\mathbf{1}%
_{\left[  t,T\right]  }\left(  r\right)  F\left(  r,X_{r},Y_{r},Z_{r}\right)
\right\vert dr\medskip\\
+\left\vert \mathbf{1}_{\left[  t_{n},T\right]  }\left(  r\right)  G\left(
r,X_{r}^{n},Y_{r}\right)  -\mathbf{1}_{\left[  t,T\right]  }\left(  r\right)
G\left(  r,X_{r},Y_{r}\right)  \right\vert dA_{r}^{n}%
\end{array}
\label{Ln}%
\end{equation}
and%
\begin{equation}
dR_{r}^{\left(  n\right)  }=\left\langle Y_{r}^{n}-Y_{r}~,\mathbf{1}_{\left[
t,T\right]  }\left(  r\right)  G\left(  r,X_{r},Y_{r}\right)  \left(
dA_{r}^{n}-dA_{r}\right)  \right\rangle \label{Rn}%
\end{equation}
Then by Lemma \ref{An_L1} below with $a=1/2$, we have
\[%
\begin{array}
[c]{r}%
\mathbb{E~}\sup\limits_{r\in\left[  0,T\right]  }e^{2\lambda\left(
r+A_{r}^{n}\right)  }\left\vert Y_{r}^{n}-Y_{r}\right\vert ^{2}+\mathbb{E}%
\left(
{\displaystyle\int_{0}^{T}}
e^{2\lambda\left(  r+A_{r}^{n}\right)  }\left\vert Z_{r}^{n}-Z_{r}\right\vert
^{2}dr\right)  \quad\quad\quad\medskip\\
\leq C_{a}~\mathbb{E}\Big[e^{2\lambda\left(  T+A_{T}^{n}\right)  }\left\vert
\kappa\left(  X_{T}^{n}\right)  -\kappa\left(  X_{T}\right)  \right\vert
^{2}+\left(
{\displaystyle\int_{0}^{T}}
e^{\lambda\left(  r+A_{r}^{n}\right)  }dL_{r}^{\left(  n\right)  }\right)
^{2}\medskip\\
+%
{\displaystyle\int_{0}^{T}}
e^{2\lambda\left(  r+A_{r}^{n}\right)  }dR_{r}^{\left(  n\right)  }\Big].
\end{array}
\]
and consequently by Lemma \ref{2a}, Lemma \ref{2} and Lemma \ref{4} below, we
have
\[
\limsup_{n\rightarrow\infty}\mathbb{E}\sup\limits_{r\in\left[  0,T\right]
}\left\vert Y_{r}^{n}-Y_{r}\right\vert ^{2}\leq\limsup_{n\rightarrow\infty
}\mathbb{E}\sup\limits_{r\in\left[  0,T\right]  }e^{2\lambda\left(
r+A_{r}^{n}\right)  }\left\vert Y_{r}^{n}-Y_{r}\right\vert ^{2}=0.
\]
We now deduce
\begin{align*}
\left\vert Y_{t_{n}}^{t_{n},x_{n}}-Y_{t}^{t,x}\right\vert ^{2} &
\leq2\mathbb{E}\left\vert Y_{t_{n}}^{t_{n},x_{n}}-Y_{t_{n}}^{t,x}\right\vert
^{2}+2\mathbb{E}\left\vert Y_{t_{n}}^{t,x}-Y_{t}^{t,x}\right\vert ^{2}\\
&  \leq2\mathbb{E}\sup\limits_{r\in\left[  0,T\right]  }\left\vert Y_{r}%
^{n}-Y_{r}\right\vert ^{2}+2\mathbb{E}\left\vert Y_{t_{n}}^{t,x}-Y_{t}%
^{t,x}\right\vert ^{2}\\
&  \rightarrow0,\quad\text{as }n\rightarrow\infty;
\end{align*}
hence the result. \hfill
\end{proof}

\bigskip

Recall that the constants $C_{1}$, $C_{2}$, $C_{3}$ and $C_{4}$ appearing in
\eqref{bk}, \eqref{ch5-pvi-4a}, \eqref{ch5-pvi-5} and \eqref{ch5-pvi-5a} are
uniform w.r.t. $(t,x)$. Consequently those estimates are valid for
$(X^{n},A^{n},Y^{n},Z^{n},U^{n},V^{n})$ for all $n\ge0$, with the same
constants, which are independent of $n$. This fact will be used repeatedly in
the proofs below.

\begin{lemma}
\label{2a} We have
\[
\lim_{n\rightarrow\infty}~\mathbb{E}~\left(  e^{2\lambda\left(  T+A_{T}%
^{n}\right)  }\left\vert \kappa\left(  X_{T}^{n}\right)  -\kappa\left(
X_{T}\right)  \right\vert ^{2}\right)  =0
\]

\end{lemma}

\begin{proof}
By Lebesgue's dominated convergence theorem and (\ref{chy5-pvi-2aa}) (also
taking in account the boundedness (\ref{chy5-pvi-2a}-jj) and (\ref{bk})), we
have
\begin{align*}
&  \mathbb{E}\left(  e^{2\lambda\left(  T+A_{T}^{n}\right)  }\left\vert
\kappa\left(  X_{T}^{n}\right)  -\kappa\left(  X_{T}\right)  \right\vert
^{2}\right) \\
&  \leq\left(  \mathbb{E}e^{4\lambda\left(  T+A_{T}^{n}\right)  }\right)
^{1/2}\left(  \mathbb{E}\left\vert \kappa\left(  X_{T}^{n}\right)
-\kappa\left(  X_{T}\right)  \right\vert ^{4}\right)  ^{1/2}\\
&  \leq C_{\lambda}~\left(  \mathbb{E}\left\vert \kappa\left(  X_{T}%
^{n}\right)  -\kappa\left(  X_{T}\right)  \right\vert ^{4}\right)  ^{1/2}\\
&  \rightarrow0,\quad\text{as }n\rightarrow\infty.
\end{align*}

\hfill
\end{proof}

\begin{lemma}
\label{2}Let $L^{\left(  n\right)  }$ defined by (\ref{Ln}). Then
\[
{\displaystyle\int_{0}^{T}} e^{\lambda\left(  r+A_{r}^{n}\right)  }%
dL_{r}^{\left(  n\right)  }\to0
\]
in mean square, as $n\to\infty$.
\end{lemma}

\begin{proof}
By (\ref{chy5-pvi-2a}-jj) we get%
\[
\mathbb{E~}\left(
{\displaystyle\int_{0}^{T}}
e^{\lambda\left(  r+A_{r}^{n}\right)  }dL_{r}^{\left(  n\right)  }\right)
^{2}\leq3~\left[  \mathbb{E}\left(  \Lambda_{n}\right)  +\mathbb{E}\left(
\Gamma_{n}\right)  +\mathbb{E}\left(  \Delta_{n}\right)  \right]  ,
\]
where%
\begin{equation}%
\begin{array}
[c]{l}%
\Lambda_{n}=\left(
{\displaystyle\int_{0}^{T}}
\left\vert \mathbf{1}_{\left[  t_{n},T\right]  }\left(  r\right)  F\left(
r,X_{r}^{n},Y_{r},Z_{r}\right)  -\mathbf{1}_{\left[  t,T\right]  }\left(
r\right)  F\left(  r,X_{r},Y_{r},Z_{r}\right)  \right\vert ^{2}dr\right)
^{2},\medskip\\
\Gamma_{n}=\left(
{\displaystyle\int_{0}^{T}}
\left\vert G\left(  r,X_{r}^{n},Y_{r}\right)  -G\left(  r,X_{r},Y_{r}\right)
\right\vert ^{2}dA_{r}^{n}\right)  ^{2},\medskip\\
\Delta_{n}=\left(
{\displaystyle\int_{0}^{T}}
\left\vert G\left(  r,X_{r},Y_{r}\right)  \right\vert ^{2}\left\vert
\mathbf{1}_{\left[  t_{n},T\right]  }\left(  r\right)  -\mathbf{1}_{\left[
t,T\right]  }\left(  r\right)  \right\vert ^{2}dA_{r}^{n}\right)  ^{2}.
\end{array}
\label{LGDn}%
\end{equation}

\textbf{Step 1.} $\mathbb{E}\left(  \Lambda_{n}\right)  \rightarrow0:$

Since%
\[
\mathbf{1}_{\left[  t_{n},T\right]  }\left(  r\right)  F\left(  r,X_{r}%
^{n},Y_{r},Z_{r}\right)  -\mathbf{1}_{\left[  t,T\right]  }\left(  r\right)
F\left(  r,X_{r},Y_{r},Z_{r}\right)  \rightarrow0\quad\text{a.e. }r\in\left[
0,T\right]  ,
\]
and%
\[%
\begin{array}
[c]{r}%
\left\vert \mathbf{1}_{\left[  t_{n},T\right]  }\left(  r\right)  F\left(
r,X_{r}^{n},Y_{r},Z_{r}\right)  -\mathbf{1}_{\left[  t,T\right]  }\left(
r\right)  F\left(  r,X_{r},Y_{r},Z_{r}\right)  \right\vert ^{2}\medskip\\
\leq C~\left(  1+\left\vert Y_{r}\right\vert ^{2}+\left\vert Z_{r}\right\vert
^{2}\right)  ,
\end{array}
\]
then by Lebesgue's dominated convergence theorem $\mathbb{E}\Lambda
_{n}\rightarrow0$.

\textbf{Step 2.} $\mathbb{E}\left(  \Gamma_{n}\right)  \rightarrow0:$

We have $\Gamma_{n}\rightarrow0,\quad\mathbb{P}-a.s.,$ because
\begin{align*}
\Gamma_{n}  &  =\left(
{\displaystyle\int_{0}^{T}}
\left\vert G\left(  r,X_{r}^{n},Y_{r}\right)  -G\left(  r,X_{r},Y_{r}\right)
\right\vert ^{2}dA_{r}^{n}\right)  ^{2}\\
&  \leq\left(  A_{T}^{n}\right)  ^{2}\sup_{r\in\left[  0,T\right]  }\left\vert
G\left(  r,X_{r}^{n},Y_{r}\right)  -G\left(  r,X_{r},Y_{r}\right)  \right\vert
^{4}.
\end{align*}
Since for all $q>1$
\begin{align*}
\mathbb{E~}\Gamma_{n}^{q}  &  \leq C~\mathbb{E}\left[  \left(  1+\left\Vert
Y\right\Vert _{T}^{4q}\right)  \left\vert A_{T}^{n}\right\vert ^{2q}\right] \\
&  \leq C_{1}\left(  1+\mathbb{E}\left\Vert Y\right\Vert _{T}^{8q}%
+\mathbb{E}\left\vert A_{T}^{n}\right\vert ^{4q}\right) \\
&  \leq C_{2},
\end{align*}
then the sequence of random variables $\Gamma_{n}$ is uniformly integrable and
therefore $\mathbb{E}\left(  \Gamma_{n}\right)  \rightarrow0.$

\textbf{Step 3.} $\mathbb{E}\left(  \Delta_{n}\right)  \rightarrow0:$

We have
\begin{align*}
\Delta_{n}  &  =\left(
{\displaystyle\int_{0}^{T}}
\left\vert G\left(  r,X_{r},Y_{r}\right)  \right\vert ^{2}\left\vert
\mathbf{1}_{\left[  t_{n},T\right]  }\left(  r\right)  -\mathbf{1}_{\left[
t,T\right]  }\left(  r\right)  \right\vert ^{2}dA_{r}^{n}\right)  ^{2}\\
&  \leq\Big(\sup_{r\in\left[  0,T\right]  }\left\vert G\left(  r,X_{r}%
,Y_{r}\right)  \right\vert ^{4}\Big)\left(
{\displaystyle\int_{0}^{T}}
\left\vert \mathbf{1}_{\left[  t_{n},T\right]  }\left(  r\right)
-\mathbf{1}_{\left[  t,T\right]  }\left(  r\right)  \right\vert ^{2}dA_{r}%
^{n}\right)  ^{2}\\
&  =\Big(\sup_{r\in\left[  0,T\right]  }\left\vert G\left(  r,X_{r}%
,Y_{r}\right)  \right\vert ^{4}\Big)\left\vert A_{t_{n}}^{n}-A_{t}%
^{n}\right\vert ^{2}\\
&  \rightarrow0,\quad\mathbb{P}-a.s.,
\end{align*}
where we have used (\ref{xnanxa}) on the last line. Moreover for $q>1$,
\begin{align*}
\mathbb{E~}\Delta_{n}^{q}  &  \leq\mathbb{E}\left[  \sup_{r\in\left[
0,T\right]  }\left\vert G\left(  r,X_{r},Y_{r}\right)  \right\vert
^{4q}~\left\vert A_{t_{n}}^{n}-A_{t}^{n}\right\vert ^{2q}\right] \\
&  \leq C~\left(  \mathbb{E}\sup_{r\in\left[  0,T\right]  }\left\vert G\left(
r,X_{r},Y_{r}\right)  \right\vert ^{8q}+\mathbb{E}\sup_{r\in\left[
0,T\right]  }\left\vert A_{r}^{n}\right\vert ^{4q}\right) \\
&  \leq C_{1}%
\end{align*}
Consequently, by uniformly integrability, we conclude that $\mathbb{E}\left(
\Delta_{n}\right)  \rightarrow0$.

\hfill
\end{proof}

\bigskip

Consider $N\in\mathbb{N}$, $N>T$ and the partition $\pi_{N}:0=r_{0}%
<r_{1}<\ldots<r_{i}<\ldots<r_{N}=T$ with $r_{i}=\frac{iT}{N}$. We denote
$\lfloor r|N\rfloor=\max\left\{  r_{i}:r_{i}\leq r\right\}  =\left[  \frac
{rN}{T}\right]  \frac{T}{N},$ where $\left[  x\right]  $ is the integer part
of $x$. Given a continuous stochastic process $\left(  H_{t}\right)
_{t\in\left[  0,T\right]  },$ we define
\[
H_{r}^{N}=\sum_{i=0}^{N-1}H_{r_{i}}\mathbf{1}_{[r_{i},r_{i+1})}\left(
r\right)  +H_{T}\mathbf{~1}_{\left\{  T\right\}  }\left(  r\right)
=H_{\lfloor r|N\rfloor}~.
\]

\begin{lemma}
\label{3}Let $1<q<2.$ There exists a positive constant $C$ independent of
$\left(  t,x\right)  ,\left(  t_{n},x_{n}\right)  \in\left[  0,T\right]
\times\overline{D}$ and $N\in\mathbb{N}$ such that
\begin{align*}
&  \limsup_{n\rightarrow\infty}\mathbb{E~}\left(  {\displaystyle\int_{0}^{T}}
\left\vert Y_{r}^{n}-Y_{r}^{n,N}\right\vert ^{q}\left(  dA_{r}^{n}%
+dA_{r}\right)  \right) \\
&  \leq\frac{C}{N^{q/2}}+C\left[  \mathbb{E~}\max_{i=\overline{1,N}}\left(
A_{r_{i}}-A_{r_{i-1}}\right)  ^{2q/\left(  2-q\right)  }\right]  ^{\left(
2-q\right)  /4}.
\end{align*}

\end{lemma}

\begin{proof}
Since%
\[%
\begin{array}
[c]{l}%
Y_{s}^{n,N}+{\displaystyle\int_{\lfloor s|N\rfloor}^{s}}\left(  U_{r}%
^{n}dr+V_{r}^{n}dA_{r}^{n}\right)  =Y_{s}^{n}+{\displaystyle\int_{\lfloor
s|N\rfloor}^{s}}\mathbf{1}_{\left[  t_{n},T\right]  }\left(  r\right)
~F\left(  r,X_{r}^{n},Y_{r}^{n},Z_{r}^{n}\right)  dr,\medskip\\
\quad+{\displaystyle\int_{\lfloor s|N\rfloor}^{s}}\mathbf{1}_{\left[
t_{n},T\right]  }\left(  r\right)  G\left(  r,X_{r}^{n},Y_{r}^{n}\right)
dA_{r}^{n}-{\displaystyle\int_{\lfloor s|N\rfloor}^{s}}\left\langle Z_{r}%
^{n},dB_{r}\right\rangle ,\ \forall~s\in\left[  0,T\right]  ,\medskip
\end{array}
\]
then
\begin{align*}
\left\vert Y_{s}^{n,N}-Y_{s}^{n}\right\vert ^{q}  &  \leq\frac{C}{N^{q/2}%
}\left[  {\displaystyle\int_{\lfloor s|N\rfloor}^{s}}\left(  \left\vert
U_{r}^{n}\right\vert ^{2}+\left\vert F\left(  r,X_{r}^{n},Y_{r}^{n},Z_{r}%
^{n}\right)  \right\vert ^{2}\right)  dr\right]  ^{q/2}\\
&  +C~\left(  A_{s}^{n}-A_{\lfloor s|N\rfloor}^{n}\right)  ^{q/2}\left[
{\displaystyle\int_{\lfloor s|N\rfloor}^{s}}\left(  \left\vert V_{r}%
^{n}\right\vert ^{2}+\left\vert G\left(  r,X_{r}^{n},Y_{r}^{n}\right)
\right\vert ^{2}\right)  dA_{r}^{n}\right]  ^{q/2}\\
&  +C~\left\vert {\displaystyle\int_{\lfloor s|N\rfloor}^{s}}\left\langle
Z_{r}^{n},dB_{r}\right\rangle \right\vert ^{q}.
\end{align*}
Hence
\[
\mathbb{E~}\left(  {\displaystyle\int_{0}^{T}}\left\vert Y_{r}^{n}-Y_{r}%
^{n,N}\right\vert ^{q}\left(  dA_{r}^{n}+dA_{r}\right)  \right)  \leq
\alpha_{n,N}+\beta_{n,N}+\gamma_{n,N}.
\]
We have first\bigskip\newline$\alpha_{n,N}=\dfrac{C}{N^{q/2}}\mathbb{E~}\Big[{%
{\displaystyle\int_{0}^{T}}
}\big({%
{\displaystyle\int_{\lfloor s|N\rfloor}^{s}}
}(\left\vert U_{r}^{n}\right\vert ^{2}+\left\vert F\left(  r,X_{r}^{n}%
,Y_{r}^{n},Z_{r}^{n}\right)  \right\vert ^{2})dr\big)^{q/2}\left(  dA_{s}%
^{n}+dA_{s}\right)  \Big]$\medskip\medskip\newline$%
\begin{array}
[c]{c}%
\
\end{array}
\leq\dfrac{C}{N^{q/2}}\mathbb{E}\Big[\left(  A_{T}^{n}+A_{T}\right)  \big({%
{\displaystyle\int_{0}^{T}}
}(\left\vert U_{r}^{n}\right\vert ^{2}+\left\vert F\left(  r,X_{r}^{n}%
,Y_{r}^{n},Z_{r}^{n}\right)  \right\vert ^{2})dr\big)^{q/2}\Big]$%
\medskip\medskip\newline$%
\begin{array}
[c]{c}%
\
\end{array}
\leq\dfrac{C}{N^{q/2}}\left[  \mathbb{E~}\left(  A_{T}^{n}+A_{T}\right)
^{\frac{2}{2-q}}\right]  ^{\frac{2-q}{2}}\left(  \mathbb{E}{\int_{0}^{T}%
}\left\vert U_{r}^{n}\right\vert ^{2}dr+\mathbb{E}{\int_{0}^{T}}\left\vert
F\left(  r,X_{r}^{n},Y_{r}^{n},Z_{r}^{n}\right)  \right\vert ^{2}dr\right)
^{\frac{q}{2}}$\medskip\medskip\newline$%
\begin{array}
[c]{c}%
\
\end{array}
\leq\dfrac{C}{N^{q/2}}.$\medskip

Since $\left(  A_{s}^{n}\right)  _{s\geq0}$ and $\left(  A_{s}\right)
_{s\geq0}$ are increasing stochastic processes, \bigskip\newline$\beta
_{n,N}=C~\mathbb{E~}{%
{\displaystyle\int_{0}^{T}}
}\Big[\left(  A_{s}^{n}-A_{\lfloor s|N\rfloor}^{n}\right)  ^{\frac{q}{2}%
}\big({%
{\displaystyle\int_{\lfloor s|N\rfloor}^{s}}
}\left(  \left\vert V_{r}^{n}\right\vert ^{2}+\left\vert G\left(  r,X_{r}%
^{n},Y_{r}^{n}\right)  \right\vert ^{2}\right)  dA_{r}^{n}\big)^{\frac{q}{2}%
}\Big]\left(  dA_{s}^{n}+dA_{s}\right)  $\medskip\medskip\newline$%
\begin{array}
[c]{c}%
\
\end{array}
\leq C~\mathbb{E}~\Big[\big(%
{\displaystyle\int_{0}^{T}}
(\left\vert V_{r}^{n}\right\vert ^{2}+\left\vert G\left(  r,X_{r}^{n}%
,Y_{r}^{n}\right)  \right\vert ^{2})dA_{r}^{n}\big)^{\frac{q}{2}}%
{\displaystyle\sum\limits_{i=1}^{N}}
{\displaystyle\int_{r_{i-1}}^{r_{i}}}
(A_{s}^{n}-A_{\lfloor s|N\rfloor}^{n})^{\frac{q}{2}}\left(  dA_{s}^{n}%
+dA_{s}\right)  \Big]$\medskip\medskip\newline$%
\begin{array}
[c]{c}%
\
\end{array}
\leq C\bigg[\mathbb{E}~\Big(%
{\displaystyle\sum\limits_{i=1}^{N}}
(A_{r_{i}}^{n}-A_{r_{i-1}}^{n})^{q/2}(A_{r_{i}}^{n}+A_{r_{i}}-A_{r_{i-1}}%
^{n}-A_{r_{i-1}})\Big)^{2/\left(  2-q\right)  }\bigg]^{\left(  2-q\right)
/2}$\bigskip

Since by (\ref{chy5-pvi-2a}-j)
\[
\lim_{n\rightarrow\infty}\mathbb{E~}\sup_{r\in\left[  0,T\right]  }\left\vert
A_{r}^{n}-A_{r}\right\vert ^{p}=0,\quad\text{for all }p>0,
\]
and
\[
\mathbb{E~}\sup_{r\in\left[  0,T\right]  }\left\vert A_{r}\right\vert
^{p}+\sup_{n\in\mathbb{N}}\left(  \mathbb{E~}\sup_{r\in\left[  0,T\right]
}\left\vert A_{r}^{n}\right\vert ^{p}\right)  <\infty,\quad\text{for all
}p>0,
\]
we infer that for all $N\in\mathbb{N}$
\begin{align*}
\limsup_{n\rightarrow\infty}\beta_{n,N}  &  \leq C\left[  \mathbb{E~}\left(
\sum_{i=1}^{N}\left(  A_{r_{i}}-A_{r_{i-1}}\right)  ^{q/2}\left(  A_{r_{i}%
}-A_{r_{i-1}}\right)  \right)  ^{2/\left(  2-q\right)  }\right]  ^{\left(
2-q\right)  /2}\\
&  \leq C~\left[  \mathbb{E~}\left(  \max_{i=\overline{1,N}}\left(  A_{r_{i}%
}-A_{r_{i-1}}\right)  ^{q/2}A_{T}\right)  ^{2/\left(  2-q\right)  }\right]
^{\left(  2-q\right)  /2}\\
&  \leq C_{1}\left[  \mathbb{E~}\max_{i=\overline{1,N}}\left(  A_{r_{i}%
}-A_{r_{i-1}}\right)  ^{2q/\left(  2-q\right)  }\right]  ^{\left(  2-q\right)
/4}.
\end{align*}
We finally consider$\medskip\medskip\newline\gamma_{n,N}=C\mathbb{E}%
{\displaystyle\int_{0}^{T}}
\left\vert {%
{\displaystyle\int_{\lfloor s|N\rfloor}^{s}}
}\left\langle Z_{r}^{n},dB_{r}\right\rangle \right\vert ^{q}\left(  dA_{s}%
^{n}+dA_{s}\right)  \medskip\medskip\newline=C~\mathbb{E~}{%
{\displaystyle\sum\limits_{i=1}^{N}}
}{%
{\displaystyle\int_{r_{i-1}}^{r_{i}}}
}\left\vert {%
{\displaystyle\int_{\lfloor s|N\rfloor}^{s}}
}\left\langle Z_{r}^{n},dB_{r}\right\rangle \right\vert ^{q}\left(  dA_{s}%
^{n}+dA_{s}\right)  \medskip\medskip\newline\leq C~{%
{\displaystyle\sum\limits_{i=1}^{N}}
}\mathbb{E}\Big[\sup\limits_{s\in\left[  r_{i-1},r_{i}\right]  }\big|{%
{\displaystyle\int_{r_{i-1}}^{s}}
}\left\langle Z_{r}^{n},dB_{r}\right\rangle \big|^{q}\left(  A_{r_{i}}%
^{n}-A_{r_{i-1}}^{n}+A_{r_{i}}-A_{r_{i-1}}\right)  \Big]\medskip
\medskip\newline\leq C~{%
{\displaystyle\sum\limits_{i=1}^{N}}
}\Big[\mathbb{E}\sup\limits_{s\in\left[  r_{i-1},r_{i}\right]  }\big|{%
{\displaystyle\int_{r_{i-1}}^{s}}
}\left\langle Z_{r}^{n},dB_{r}\right\rangle \big|^{2}\Big]^{q/2}%
\mathbb{~}\left[  \mathbb{E~}\left(  A_{r_{i}}^{n}-A_{r_{i-1}}^{n}+A_{r_{i}%
}-A_{r_{i-1}}\right)  ^{\frac{2}{2-q}}\right]  ^{\frac{2-q}{2}}\medskip
\medskip\newline\leq C_{1}~{%
{\displaystyle\sum\limits_{i=1}^{N}}
}\left(  \mathbb{E}{%
{\displaystyle\int_{r_{i-1}}^{r_{i}}}
}\left\vert Z_{r}^{n}\right\vert ^{2}dr\right)  ^{q/2}\mathbb{~}\left[
\mathbb{E~}\left(  A_{r_{i}}^{n}-A_{r_{i-1}}^{n}+A_{r_{i}}-A_{r_{i-1}}\right)
^{\frac{2}{2-q}}\right]  ^{\frac{2-q}{2}}$ .$\medskip\medskip\newline$From the
above and the following H\"{o}lder's inequality, for $1<q<2$,
\[
{\displaystyle\sum\limits_{i=1}^{N}}a_{i}^{q/2}b_{i}^{\left(  2-q\right)
/2}\leq\left(  {\displaystyle\sum\limits_{i=1}^{N}}a_{i}\right)  ^{q/2}\left(
{\displaystyle\sum\limits_{i=1}^{N}}b_{i}\right)  ^{\left(  2-q\right)  /2},
\]
we deduce that
\[
\gamma_{n,N}\leq C_{2}\left[  {\displaystyle\sum\limits_{i=1}^{N}}%
\mathbb{E~}\left(  A_{r_{i}}^{n}-A_{r_{i-1}}^{n}+A_{r_{i}}-A_{r_{i-1}}\right)
^{2/\left(  2-q\right)  }\right]  ^{\left(  2-q\right)  /2}.
\]
Hence for all $N\in\mathbb{N}$
\begin{align*}
\limsup_{n\rightarrow\infty}\gamma_{n,N}  &  \leq C~\left[  {\displaystyle\sum
\limits_{i=1}^{N}}\mathbb{E~}\left(  A_{r_{i}}-A_{r_{i-1}}\right)  ^{2/\left(
2-q\right)  }\right]  ^{\left(  2-q\right)  /2}\\
&  \leq C~\left[  \mathbb{E}\left(  \max_{i=\overline{1,N}}\left(  A_{r_{i}%
}-A_{r_{i-1}}\right)  ^{q/\left(  2-q\right)  }\sum_{i=1}^{N}\left(  A_{r_{i}%
}-A_{r_{i-1}}\right)  \right)  \right]  ^{\left(  2-q\right)  /2}\\
&  \leq C_{1}\left[  \mathbb{E~}\max_{i=\overline{1,N}}\left(  A_{r_{i}%
}-A_{r_{i-1}}\right)  ^{2q/\left(  2-q\right)  }\right]  ^{\left(  2-q\right)
/4}.
\end{align*}
The result follows. \hfill
\end{proof}

\begin{lemma}
\label{4}Let $R^{\left(  n\right)  }$ defined by (\ref{Rn}). Then%
\[
\limsup_{n\rightarrow\infty}\mathbb{E}%
{\displaystyle\int_{0}^{T}}
e^{2\lambda\left(  r+A_{r}^{n}\right)  }dR_{r}^{\left(  n\right)  }=0.
\]

\end{lemma}

\begin{proof}
Denote $G_{r}=G\left(  r,X_{r},Y_{r}\right)  $ and $\left\Vert G\right\Vert
_{T}=\sup_{r\in\left[  0,T\right]  }\left\vert G_{r}\right\vert .$ Then%
\begin{align*}
\left(  Y_{r}^{n}-Y_{r}\right)  G\left(  r,X_{r},Y_{r}\right)   &  =\left(
Y_{r}^{n,N}-Y_{r}^{N}\right)  \left(  G_{r}-G_{r}^{N}\right)  +\left(
Y_{r}^{N}-Y_{r}\right)  G_{r}\\
&  +\left(  Y_{r}^{n,N}-Y_{r}^{N}\right)  G_{r}^{N}+\left(  Y_{r}^{n}%
-Y_{r}^{n,N}\right)  G_{r}%
\end{align*}
and therefore
\begin{align*}
&  \mathbb{E~}\left(  {\displaystyle\int_{0}^{T}}e^{2\lambda\left(
r+A_{r}^{n}\right)  }dR_{r}^{\left(  n\right)  }\right) \\
&  =\mathbb{E}{\displaystyle\int_{0}^{T}}e^{2\lambda\left(  r+A_{r}%
^{n}\right)  }\left(  Y_{r}^{n}-Y_{r}\right)  \mathbf{1}_{\left[  t,T\right]
}\left(  r\right)  G\left(  r,X_{r},Y_{r}\right)  \left(  dA_{r}^{n}%
-dA_{r}\right) \\
&  \leq(2\lambda)^{-1}\mathbb{E}\left[  \left(  \left(  \left\Vert
Y^{n}\right\Vert _{T}+\left\Vert Y\right\Vert _{T}\right)  \left\Vert
G-G^{N}\right\Vert _{T}+\left\Vert Y^{N}-Y\right\Vert _{T}\left\Vert
G\right\Vert _{T}\right)  e^{2\lambda\left(  T+A_{T}^{n}+A_{T}\right)
}\right] \\
&  +\mathbb{E~}\left(  e^{2\lambda(T+A_{T}^{n})}{\displaystyle\sum
\limits_{i=1}^{N}}\left(  Y_{r_{i-1}}^{n}-Y_{r_{i-1}}\right)  G_{r_{i-1}%
}\left[  \left(  A_{r_{i}}^{n}-A_{r_{i}}\right)  -\left(  A_{r_{i-1}}%
^{n}-A_{r_{i-1}}\right)  \right]  \right) \\
&  +\mathbb{E~}\left(  e^{2\lambda\left(  T+A_{T}^{n}\right)  }\left\Vert
G\right\Vert _{T}{\displaystyle\int_{0}^{T}}\left\vert Y_{r}^{n}-Y_{r}%
^{n,N}\right\vert \left(  dA_{r}^{n}+dA_{r}\right)  \right)
\end{align*}
Let $1<q<2.$ Using H\"{o}lder's inequality and the estimates (\ref{ch5-pvi-4a}%
) and (\ref{ch5-pvi-5a}), we obtain
\begin{align*}
\mathbb{E~}\left(  {\displaystyle\int_{0}^{T}}e^{2\lambda\left(  r+A_{r}%
^{n}\right)  }dR_{r}^{\left(  n\right)  }\right)   &  \leq C~\sqrt
{\mathbb{E}\left\Vert G-G^{N}\right\Vert _{T}^{2}}+\sqrt{\mathbb{E}\left\Vert
Y^{N}-Y\right\Vert _{T}^{2}}\\
&  +C{\displaystyle\sum\limits_{i=1}^{N}}\left[  \mathbb{E~}\left\vert \left(
A_{r_{i}}^{n}-A_{r_{i}}\right)  -\left(  A_{r_{i-1}}^{n}-A_{r_{i-1}}\right)
\right\vert ^{2}\right]  ^{1/2}\\
&  +C\left(  \mathbb{E~}{\displaystyle\int_{0}^{T}}\left\vert Y_{r}^{n}%
-Y_{r}^{n,N}\right\vert ^{q}\left(  dA_{r}^{n}+dA_{r}\right)  \right)  ^{1/q}%
\end{align*}
By Lemma \ref{3} we deduce that for all $N\in\mathbb{N}$
\begin{align*}
\limsup_{n\rightarrow\infty}\mathbb{E}{\displaystyle\int_{0}^{T}}%
e^{2\lambda\left(  r+A_{r}^{n}\right)  }dR_{r}^{\left(  n\right)  }  &  \leq
C~\sqrt{\mathbb{E}\left\Vert G-G^{N}\right\Vert _{T}^{2}}+\sqrt{\mathbb{E}%
\left\Vert Y^{N}-Y\right\Vert _{T}^{2}}\\
&  +C\left[  \frac{1}{N^{q/2}}+\left[  \mathbb{E~}\max_{i=\overline{1,N}%
}\left(  A_{r_{i}}-A_{r_{i-1}}\right)  ^{2q/\left(  2-q\right)  }\right]
^{\left(  2-q\right)  /4}\right]  ^{1/q}%
\end{align*}
and the result follows passing to limit as $N\rightarrow\infty$ in the last inequality.

\hfill
\end{proof}

Theorem \ref{th:main} in the particular case $\varphi=\psi\equiv0$ yields the following

\begin{corollary}
Proposition 4.1 from \cite{pa-zh/98} and Corollary 14 from \cite{ma-ra/10}
hold true.
\end{corollary}

\bigskip

\section{Infinite horizon BSDEs: continuity}

Let us consider the forward-backward problem (\ref{ch5-pvi-2}) \&
(\ref{ch5-pvi-3}) on the interval $[0,\infty)$ with $f,g,F$ and $G$
independent of time argument, $\kappa=0$ and $\varphi=\psi\equiv0,$ $u_{0}=0,$
that is:

the forward reflected SDE starting from $x\ $at $t=0:$
\[%
\begin{array}
[c]{rl}%
\left(  j\right)  \; & X_{s}^{x}\in\overline{D}\;\text{ for all }%
s\geq0,\medskip\\
\left(  jj\right)  \; & 0=A_{0}^{x}\leq A_{s}^{x}\leq A_{u}^{x}\text{ for all
}0\leq s\leq u,\medskip\\
\left(  jjj\right)  \; & X_{s}^{x}+%
{\displaystyle\int_{0}^{s}}
\nabla\phi\left(  X_{r}^{x}\right)  dA_{r}^{x}=x+%
{\displaystyle\int_{0}^{s}}
f\left(  X_{r}^{x}\right)  dr\smallskip\\
& \multicolumn{1}{r}{+%
{\displaystyle\int_{0}^{s}}
g\left(  X_{r}^{x}\right)  dB_{r},\ \;\forall~s\geq0,\medskip}\\
\left(  jv\right)  \; & A_{s}^{x}=%
{\displaystyle\int_{0}^{s}}
\mathbf{1}_{Bd\left(  \overline{D}\right)  }\left(  X_{r}^{x}\right)
dA_{r}^{x}~,\;~\forall~s\geq0.
\end{array}
\]
and the BSDE on $[0,\infty)$ with the final data $0:$%
\begin{equation}
Y_{s}^{x}=%
{\displaystyle\int_{s}^{\infty}}
F\left(  X_{r}^{x},Y_{r}^{x},Z_{r}^{x}\right)  dr+%
{\displaystyle\int_{s}^{\infty}}
G\left(  X_{r}^{x},Y_{r}^{x}\right)  dA_{r}^{x}-%
{\displaystyle\int_{s}^{\infty}}
Z_{r}^{x}dB_{r},\;s\geq0, \label{bsde-iii}%
\end{equation}
Denote $\left(  X_{s}^{x},A_{s}^{x},Y_{s}^{x;n},Z_{s}^{x;n}\right)  =\left(
X_{s}^{0,x},A_{s}^{0,x},Y_{s}^{0,x},Z_{s}^{0,x}\right)  ,$ $n\in\mathbb{N},$
the solution of the forward-backward problem (\ref{ch5-pvi-2}%
)\&(\ref{ch5-pvi-3}) on the time interval $\left[  0,n\right]  $ with $\left(
Y_{s}^{x;n},Z_{s}^{x;n}\right)  =0,\;$for $s>n;$ hence%
\begin{equation}
Y_{s}^{x;n}=%
{\displaystyle\int_{s}^{n}}
F\left(  X_{r}^{x},Y_{r}^{x;n},Z_{r}^{x;n}\right)  dr+%
{\displaystyle\int_{s}^{n}}
G\left(  X_{r}^{x},Y_{r}^{x;n}\right)  dA_{r}^{x}-%
{\displaystyle\int_{s}^{n}}
Z_{r}^{x;n}dB_{r},\;s\in\left[  0,n\right]  , \label{eli-approx}%
\end{equation}
By Theorem \ref{th:main} the mapping
\begin{equation}
x\longmapsto Y_{0}^{x;n}:\overline{D}\rightarrow\mathbb{R}^{m}\text{ is
continuous}. \label{yxncont}%
\end{equation}

Estimates on the approximating equation (\ref{eli-approx}) and the continuity
result (\ref{yxncont}) yield:

\begin{proposition}
\label{el-p}Under the assumptions (\ref{h3-0}) and $\max\left\{  \left(
\mu_{F}+\ell_{F}^{2}\right)  ,\mu_{G}\right\}  \leq\lambda<0$ there exists a
unique pair $\left(  Y^{x},Z^{x}\right)  \in S_{m}^{0}\left[  0,T\right]
\times\Lambda_{m\times k}^{0}\left(  0,T\right)  $ solution of the BSDE
(\ref{bsde-iii}) in the following sense:%
\begin{equation}
\left\{
\begin{array}
[c]{rr}%
\left(  j\right)  \; & Y_{s}^{x}=Y_{T}^{x}+%
{\displaystyle\int_{s}^{T}}
F\left(  X_{r}^{x},Y_{r}^{x},Z_{r}^{x}\right)  dr+%
{\displaystyle\int_{s}^{T}}
G\left(  X_{r}^{x},Y_{r}^{x}\right)  dA_{r}^{x}-%
{\displaystyle\int_{s}^{T}}
Z_{r}^{x}dB_{r},\smallskip\\
& \;\text{for all }0\leq s\leq T,\medskip\\
\left(  jj\right)  \; & \multicolumn{1}{l}{\mathbb{E~}\sup\limits_{r\geq
0}e^{2\lambda(r+A_{r}^{x})}\left\vert Y_{r}^{x}\right\vert ^{2}+\mathbb{E~}{%
{\displaystyle\int_{0}^{\infty}}
}e^{2\lambda(r+A_{r}^{x})}\left\vert Z_{r}^{x}\right\vert ^{2}dr<\infty
,\medskip}\\
\left(  jjj\right)  \; & \multicolumn{1}{l}{\lim\limits_{N\rightarrow\infty
}\mathbb{E~}\sup\limits_{r\geq N}e^{2\lambda(r+A_{r}^{x})}\left\vert Y_{r}%
^{x}\right\vert ^{2}=0.}%
\end{array}
\right.  \label{bsde-el1}%
\end{equation}
Moreover the mapping%
\begin{equation}
x\longmapsto u\left(  x\right)  =Y_{0}^{x}:\overline{D}\rightarrow
\mathbb{R}^{m}\text{ is continuous.} \label{el-cont}%
\end{equation}

\end{proposition}

\begin{proof}
The existence and uniqueness result for the solution of (\ref{bsde-el1}) was
proved by Pardoux and Zhang in \cite{pa-zh/98}, Theorem 2.1 (the result is
also given in \cite{pa-ra/14}, Section 5.6.1). Proving here the continuity
property (\ref{el-cont}) we obtain, once again, the existence of the solution;
the uniqueness is a easy consequence of Lemma \ref{An_L1} via the assumptions
(\ref{h3-0}) on $F$ and $G$.

Using (\ref{h3-0}) we also deduce by Lemma \ref{An_L1} with $a=1/2$ (or
directly from (\ref{ch5-pvi-4})) that for $0\leq s\leq n:$%
\[%
\begin{array}
[c]{l}%
\mathbb{E~}\sup\limits_{r\in\left[  s,n\right]  }e^{2\lambda(r+A_{r}^{x}%
)}\left\vert Y_{r}^{x;n}\right\vert ^{2}+\mathbb{E~}{%
{\displaystyle\int_{s}^{n}}
}e^{2\lambda(r+A_{r}^{x})}\left\vert Z_{r}^{x;n}\right\vert ^{2}dr\medskip\\
\quad\quad\leq C~\mathbb{E}\bigg[e^{2\lambda(n+A_{n}^{x})}\left\vert
Y_{n}^{x;n}\right\vert ^{2}+\Big(%
{\displaystyle\int_{s}^{n}}
e^{\lambda(r+A_{r}^{x})}\left\vert F\left(  X_{r}^{x},0,0\right)  \right\vert
dr\Big)^{2}\medskip\\
\quad\quad+\Big(%
{\displaystyle\int_{s}^{n}}
e^{\lambda(r+A_{r}^{x})}\left\vert G\left(  X_{r}^{x},0\right)  \right\vert
dA_{r}^{x}\Big)^{2}\bigg]\medskip\\
\quad\quad\leq C^{\prime}~\mathbb{E}~\left(
{\displaystyle\int_{s}^{n}}
e^{\lambda\left(  r+A_{r}^{x}\right)  }(dr+dA_{r}^{x})\right)  ^{2}\medskip\\
\quad\quad\leq\dfrac{C^{\prime}}{\left\vert \lambda\right\vert }%
~\mathbb{E~}e^{2\lambda(s+A_{s}^{x})}\\
\quad\quad\leq\dfrac{C^{\prime}}{\left\vert \lambda\right\vert }~e^{2\lambda
s},
\end{array}
\]
(we also used that $F(X_{r}^{x},0,0)$ and $G(X_{r}^{x},0)$ are uniformly
bounded on the bounded domain $\overline{D}$ ).

Since $\left(  Y_{s}^{x;n},Z_{s}^{x;n}\right)  =0,\;$for $s>n$ we infer that
for all $s\geq0$ and $n\in\mathbb{N},$%
\begin{equation}
\mathbb{E~}\sup\limits_{r\geq s}e^{2\lambda(r+A_{r}^{x})}\left\vert
Y_{r}^{x;n}\right\vert ^{2}+\mathbb{E~}{\int_{s}^{\infty}}e^{2\lambda
(r+A_{r}^{x})}\left\vert Z_{r}^{x;n}\right\vert ^{2}dr\leq\frac{C}{\left\vert
\lambda\right\vert }~e^{2\lambda s}.\label{estim_yxn}%
\end{equation}
If $n,l\in\mathbb{N}$ and $s\in\left[  0,n\right]  ,$ then
\[
Y_{s}^{x;n+l}-Y_{s}^{x;n}=Y_{n}^{x;n+l}+{\displaystyle\int_{s}^{n}%
}d\mathcal{K}_{r}-{\displaystyle\int_{s}^{n}}\left(  Z_{r}^{x;n+l}-Z_{r}%
^{x;n}\right)  dB_{r},
\]
where
\begin{align*}
d\mathcal{K}_{r} &  =\left[  F(X_{r}^{x},Y_{r}^{x;n+l},Z_{r}^{x;n+l})-F\left(
X_{r}^{x},Y_{r}^{x;n},Z_{r}^{x;n}\right)  \right]  dr\\
&  -\left[  G(X_{r}^{x},Y_{r}^{x;n+l})-G\left(  X_{r}^{x},Y_{r}^{x;n}\right)
\right]  dA_{r}^{x}.
\end{align*}
By the assumptions (\ref{h3-0}) we have%
\[%
\begin{array}
[c]{l}%
\left\langle Y_{r}^{x;n+l}-Y_{r}^{x;n}~,d\mathcal{K}_{r}\right\rangle
\medskip\\
\multicolumn{1}{r}{\leq\mu_{F}\left\vert Y_{r}^{x;n+l}-Y_{r}^{x;n}\right\vert
^{2}dr+\ell_{F}\left\vert Y_{r}^{x;n+l}-Y_{r}^{x;n}\right\vert \left\vert
Z_{r}^{x;n+l}-Z_{r}^{x;n}\right\vert dr\smallskip}\\
\multicolumn{1}{r}{+\mu_{G}\left\vert Y_{r}^{x;n+l}-Y_{r}^{x;n}\right\vert
^{2}dA_{r}^{x}\medskip}\\
\leq\dfrac{1}{4}\left\vert Z_{r}^{x;n+l}-Z_{r}^{x;n}\right\vert ^{2}%
dr+\left\vert Y_{r}^{x;n+l}-Y_{r}^{x;n}\right\vert ^{2}\lambda\left(
dr+dA_{r}^{n}\right)  .
\end{array}
\]
Therefore by Lemma \ref{An_L1} (with $a=1/2$) and (\ref{estim_yxn}) we get
\[%
\begin{array}
[c]{l}%
\mathbb{E~}\sup\limits_{r\in\left[  0,n\right]  }e^{2\lambda\left(
r+A_{r}^{x}\right)  }\left\vert Y_{r}^{x;n+l}-Y_{r}^{x;n}\right\vert
^{2}+\mathbb{E~}%
{\displaystyle\int_{0}^{n}}
e^{2\lambda\left(  r+A_{r}^{x}\right)  }\left\vert Z_{r}^{x;n+l}-Z_{r}%
^{x;n}\right\vert ^{2}dr\medskip\\
\quad\quad\leq C~\mathbb{E}~e^{2\lambda\left(  n+A_{n}^{x}\right)  }\left\vert
Y_{n}^{x;n+l}\right\vert ^{2}\medskip\\
\quad\quad\leq\dfrac{C}{\left\vert \lambda\right\vert }~e^{2\lambda n}.
\end{array}
\]
Hence%
\[
\mathbb{E~}\sup\limits_{r\geq0}e^{2\lambda\left(  r+A_{r}^{x}\right)
}\left\vert Y_{r}^{x;n+l}-Y_{r}^{x;n}\right\vert ^{2}+\mathbb{E~}%
{\displaystyle\int_{0}^{\infty}}
e^{2\lambda\left(  r+A_{r}^{x}\right)  }\left\vert Z_{r}^{x;n+l}-Z_{r}%
^{x;n}\right\vert ^{2}dr\leq\dfrac{C}{\left\vert \lambda\right\vert
}~e^{2\lambda n}%
\]
and consequently there exists $\left(  Y_{s}^{x},Z_{s}^{x}\right)  _{s\geq0}$
a pair of progressively measurable stochastic process, $\left(  Y_{s}%
^{x}\right)  _{s\geq0}$ having continuous trajectories, such that for all
$s\geq0$%
\[
\mathbb{E~}\sup\limits_{r\geq s}e^{2\lambda(r+A_{r}^{x})}\left\vert Y_{r}%
^{x}\right\vert ^{2}+\mathbb{E~}{\int_{s}^{\infty}}e^{2\lambda(r+A_{r}^{x}%
)}\left\vert Z_{r}^{x}\right\vert ^{2}dr<\frac{C}{\left\vert \lambda
\right\vert }~e^{2\lambda s}%
\]
and%
\begin{align*}
&  \mathbb{E~}\sup\limits_{r\geq0}e^{2\lambda\left(  r+A_{r}^{x}\right)
}\left\vert Y_{r}^{x}-Y_{r}^{x;n}\right\vert ^{2}+\mathbb{E~}%
{\displaystyle\int_{0}^{\infty}}
e^{2\lambda\left(  r+A_{r}^{x}\right)  }\left\vert Z_{r}^{x}-Z_{r}%
^{x;n}\right\vert ^{2}dr\medskip\\
&  \leq\dfrac{C}{\left\vert \lambda\right\vert }~e^{2\lambda n}\rightarrow
0,\quad\text{as }n\rightarrow\infty.
\end{align*}
Since for all $0\leq T\leq n:$%
\[
Y_{s}^{x;n}=Y_{T}^{x;n}+%
{\displaystyle\int_{s}^{T}}
F\left(  X_{r}^{x},Y_{r}^{x;n},Z_{r}^{x;n}\right)  dr+%
{\displaystyle\int_{s}^{T}}
G\left(  X_{r}^{x},Y_{r}^{x;n}\right)  dA_{r}^{x}-%
{\displaystyle\int_{s}^{T}}
Z_{r}^{x;n}dB_{r},\;s\in\left[  0,n\right]  ,
\]
then passing to limit as $n\rightarrow\infty$ (possibly along a subsequence)
we obtain that $\left(  Y_{s}^{x},Z_{s}^{x}\right)  _{s\geq0}$ is a solution
of (\ref{bsde-el1}).

Let $y,x\in\overline{D}.$ Since%
\begin{align*}
\left\vert Y_{0}^{y}-Y_{0}^{x}\right\vert  &  \leq\left\vert Y_{0}^{y}%
-Y_{0}^{y;n}\right\vert +\left\vert Y_{0}^{y;n}-Y_{0}^{x;n}\right\vert
+\left\vert Y_{0}^{x;n}-Y_{0}^{x}\right\vert \\
&  \leq\tfrac{2~\sqrt{C}}{\sqrt{\left\vert \lambda\right\vert }}~e^{\lambda
n}+\left\vert Y_{0}^{y;n}-Y_{0}^{x;n}\right\vert ,\;\;\text{for all }%
n\in\mathbb{N}.
\end{align*}
and $\lambda<0$, the continuity property (\ref{el-cont}) follows from
(\ref{yxncont}).

\hfill
\end{proof}

We finally deduce that

\begin{remark}
Theorem 5.1 from \cite{pa-zh/98} holds true.
\end{remark}

\section{Viscosity solutions\label{s4}}

\subsection{Parabolic PDEs}

We recall  some results on the viscosity solutions of the PVI
(\ref{ch5-pvi-1}) from \cite{Pa-Ra/98}, \cite{Ma/Pa/Ra/Za:09}, \cite{ma-ra/10},
\cite{pa-ra/14}. At the same time, we formulate the definition of the notion of viscosity solution
of our system of equations.

We assume that the assumptions from Section \ref{s1} and Section \ref{s2} are
satisfied and we let the dimension of the Brownian motion be $k=d.$

Denote $\mathbb{S}^{d}$ the set of symmetric matrices from $\mathbb{R}%
^{d\times d}.$

Let $h:\left[  0,T\right]  \times\overline{D}\rightarrow\mathbb{R}$ be a
continuous function.

A triple $(p,q,X)\in\mathbb{R}\times\mathbb{R}^{d}\times\mathbb{S}^{d}$ is a
parabolic super-jet to $h$, at $\left(  t,x\right)  \in\left[  0,T\right]
\times\overline{D},$ if for all $\left(  s,x^{\prime}\right)  \in\left[
0,T\right]  \times\overline{D},$%
\begin{equation}%
\begin{array}
[c]{r}%
h(s,x^{\prime})\leq h(t,x)+p(s-t)+\langle q,x^{\prime}-x\rangle+\dfrac{1}%
{2}\langle X(x^{\prime}-x),x^{\prime}-x\rangle\medskip\\
+o(|s-t|+|x^{\prime}-x|^{2}).
\end{array}
\label{sup-jet}%
\end{equation}
The set of parabolic super-jets at $\left(  t,x\right)  $ is denoted by
$\mathcal{P}^{2,+}h(t,x);$ the set of parabolic sub-jets is defined by
$\mathcal{P}_{\mathcal{O}}^{2,-}h=-\mathcal{P}_{\mathcal{O}}^{2,+}(-h)$.

First we consider the system (\ref{ch5-pvi-1}) with the functions
$\varphi,\psi:\mathbb{R}^{m}\rightarrow]-\infty,+\infty]$ decoupled on
coordinates as follows $\varphi\left(  u_{1},\ldots,u_{m}\right)  ={\varphi
}_{1}\left(  u_{1}\right)  +\cdots+{\varphi}_{m}\left(  u_{m}\right)  $ and
$\psi\left(  u_{1},\ldots,u_{m}\right)  ={\psi}_{1}\left(  u_{1}\right)
+\cdots+{\psi}_{m}\left(  u_{m}\right)  ,$ where $\varphi_{i},\psi
_{i}:\mathbb{R}\rightarrow]-\infty,+\infty]$ are l.s.c. convex functions;
hence ${\partial\varphi}\left(  u_{1},\ldots,u_{m}\right)  ={\partial\varphi
}_{1}\left(  u_{1}\right)  \times\cdots\times{\partial\varphi}_{m}\left(
u_{m}\right)  $ and similar for $\partial\psi.$

We also assume that $F_{i}$ , the $i-$th coordinate of $F$, depends only on
the $i-$th row of the matrix $Z$.

Consider the system%
\begin{equation}
\left\{
\begin{array}
[c]{rr}%
\left(  a\right)  \; & -\dfrac{\partial u_{i}(t,x)}{\partial t}-\mathcal{L}%
_{t}u_{i}\left(  t,x\right)  +{\partial\varphi}_{i}\big(u_{i}(t,x)\big)\ni
F_{i}\big(t,x,u(t,x),\left(  \nabla u_{i}(t,x)\right)  ^{\ast}%
g(t,x)\big),\smallskip\\
& t\in\left(  0,T\right)  ,\;x\in D,\quad i\in\overline{1,m},\medskip\\
\multicolumn{1}{l}{\left(  b\right)  \;} & \multicolumn{1}{l}{\dfrac{\partial
u_{i}(t,x)}{\partial n}+{\partial\psi}_{i}\big(u_{i}(t,x)\big)\ni
G_{i}\big(t,x,u(t,x)\big),}\\
\multicolumn{1}{l}{} & \multicolumn{1}{l}{\quad\quad\quad\quad\quad\quad
\quad\quad\quad\quad\quad\quad\quad\quad t\in\left(  0,T\right)  ,\;x\in
Bd\left(  \overline{D}\right)  ,\quad i\in\overline{1,m},\medskip}\\
\multicolumn{1}{l}{\left(  c\right)  \;} & \multicolumn{1}{l}{u(T,x)=\kappa
(x),\;\ x\in\overline{{D}},}%
\end{array}
\right.  \label{pvi1}%
\end{equation}

where%
\[
\mathcal{L}_{t}u_{i}\left(  t,x\right)  =\dfrac{1}{2}%
{\displaystyle\sum\limits_{j,l=1}^{d}}
\left(  gg^{\ast}\right)  _{j,l}(t,x)\dfrac{\partial^{2}u_{i}(t,x)}{\partial
x_{j}\partial x_{l}}+%
{\displaystyle\sum\limits_{j=1}^{d}}
f_{j}\left(  t,x\right)  \dfrac{\partial u_{i}(t,x)}{\partial x_{j}}%
\]
$\bigskip$

Define $\Phi_{i},\Gamma_{i}:\left[  0,T\right]  \times\overline{D}%
\times\mathbb{R}^{m}\times\mathbb{R}^{d}\times\mathbb{S}^{d}\rightarrow
\mathbb{R}$, $i\in\overline{1,m},$ to be the functions:%
\begin{equation}%
\begin{array}
[c]{rl}%
\Phi_{i}\left(  t,x,y,q,X\right)  = & \dfrac{1}{2}\mathrm{Tr}\big((gg^{\ast
})(t,x)X\big)+\langle q,f(t,x)\rangle+F_{i}\big(t,x,y,q^{\ast}%
g(t,x)\big)\medskip\\
\Gamma_{i}(t,x,y,q)= & -\langle\nabla\phi(x),q\rangle+G_{i}(t,x,y).
\end{array}
\label{figi}%
\end{equation}

If $u=\left(  u_{1},\ldots,u_{m}\right)  ^{\ast}:\left[  0,T\right]
\times\overline{D}\rightarrow\mathbb{R}^{m}$, then for each $i\in
\overline{1,m}$ we have%
\begin{align*}
\Phi_{i}\left(  t,x,u\left(  t,x\right)  ,\nabla u_{i}(t,x),D^{2}%
u_{i}(t,x)\right)   &  =\mathcal{L}_{t}u_{i}(t,x)+F_{i}\big(t,x,u\left(
t,x\right)  ,\left(  \nabla u_{i}(t,x)\right)  ^{\ast}g(t,x)\big),\;\text{and}%
\medskip\\
\Gamma_{i}(t,x,u\left(  t,x\right)  ,\nabla u_{i}(t,x))  &  =-\frac{\partial
u_{i}(t,x)}{\partial n}+G_{i}(t,x,u\left(  t,x\right)  ).
\end{align*}

We put the notations  $a\wedge b\overset{def}{=}\min\left\{  a,b\right\}  $
and $a\vee b\overset{def}{=}\max\left\{  a,b\right\}  .$

The following results hold.

\begin{theorem}
\label{PZ}(Pardoux, Zhang \cite{pa-zh/98}: Theorem 4.3; Pardoux,
R\u{a}\c{s}canu \cite{pa-ra/14} : Theorem 5.43) Consider the parabolic system
(\ref{pvi1}) with $\varphi=\psi=0.$ Then the continuous function $u:\left[
0,T\right]  \times\overline{{D}}\rightarrow\mathbb{R}^{m}$ defined by
(\ref{ch5-pvi-6}) is a viscosity solution of the parabolic partial
differential system (\ref{pvi1}) i.e.
\[
\text{\ }u(T,x)=\kappa\left(  x\right)  ,\;\forall~x\in\overline{{D}},
\]
and $u$ is a viscosity sub-solution that is, for any $i\in\overline{1,m}$ :%
\[%
\begin{array}
[c]{ll}%
\left(  a\right)  \quad & \text{for any }\left(  t,x\right)  \in
(0,T)\times\overline{{D}}\text{, any }(p,q,X)\in\mathcal{P}^{2,+}%
u_{i}(t,x):\smallskip\\
& \quad\quad\quad\quad\quad\quad p+\Phi_{i}\left(  t,x,u\left(  t,x\right)
,q,X\right)  \geq0,\;\medskip\\
\left(  b\right)  \quad & \text{for any }\left(  t,x\right)  \in(0,T)\times
Bd\left(  \overline{{D}}\right)  \text{, any }(p,q,X)\in\mathcal{P}^{2,+}%
u_{i}(t,x):\smallskip\\
& \quad\quad\quad\quad\quad\quad\left[  p+\Phi_{i}\left(  t,x,u\left(
t,x\right)  ,q,X\right)  \right]  \vee\Gamma_{i}\left(  t,x,u\left(
t,x\right)  ,q\right)  \geq0,
\end{array}
\]
together with $u$ is a viscosity super-solution that is, for any
$i\in\overline{1,m}$ :%
\[%
\begin{array}
[c]{ll}%
\left(  c\right)  \quad & \text{for any }\left(  t,x\right)  \in
(0,T)\times\overline{{D}}\text{, any }(p,q,X)\in\mathcal{P}^{2,-}%
u_{i}(t,x):\smallskip\\
& \quad\quad\quad\quad\quad\quad p+\Phi_{i}\left(  t,x,u\left(  t,x\right)
,q,X\right)  \leq0,\;\medskip\\
\left(  d\right)  \quad & \text{for any }\left(  t,x\right)  \in(0,T)\times
Bd\left(  \overline{{D}}\right)  \text{, any }(p,q,X)\in\mathcal{P}^{2,-}%
u_{i}(t,x):\smallskip\\
& \quad\quad\quad\quad\quad\quad\left[  p+\Phi_{i}\left(  t,x,u\left(
t,x\right)  ,q,X\right)  \right]  \wedge\Gamma_{i}\left(  t,x,u\left(
t,x\right)  ,q\right)  \leq0.
\end{array}
\]

\end{theorem}

\begin{theorem}
\label{MR}(Maticiuc, R\u{a}\c{s}canu \cite{ma-ra/10}: Theorem 5; Pardoux,
R\u{a}\c{s}canu \cite{pa-ra/14} : Theorem 5.81) The continuous function
$u:\left[  0,T\right]  \times\overline{{D}}\rightarrow\mathbb{R}^{m}$ defined
by (\ref{ch5-pvi-6}) is a viscosity solution of the parabolic differential
system (\ref{pvi1}) on $\overline{D}$ i.e.
\[
\left\vert \text{\ }%
\begin{array}
[c]{l}%
u(T,x)=\kappa\left(  x\right)  ,\;\forall~x\in\overline{{D}},\medskip\\
u(t,x)\in Dom\left(  \varphi\right)  ,\ \ \forall{(t,x)}\in(0,T)\times
\overline{{D}},\medskip\\
u(t,x)\in Dom\left(  \psi\right)  ,\ \ \ \forall{(t,x)}\in(0,T)\times
Bd\left(  \overline{D}\right)  ,
\end{array}
\right.
\]
and $u$ is a viscosity sub-solution that is, for any $i\in\overline{1,m}$ :%
\[%
\begin{array}
[c]{ll}%
\left(  a\right)  \quad & \text{for any }\left(  t,x\right)  \in
(0,T)\times\overline{{D}}\text{, any }(p,q,X)\in\mathcal{P}^{2,+}%
u_{i}(t,x):\smallskip\\
& \quad\quad\quad\quad\quad\quad p+\Phi_{i}\left(  t,x,u\left(  t,x\right)
,q,X\right)  \geq\left(  \varphi_{i}\right)  _{-}^{\prime}\big(u_{i}\left(
t,x\right)  \big),\medskip\\
\left(  b\right)  \quad & \text{for any }\left(  t,x\right)  \in(0,T)\times
Bd\left(  \overline{{D}}\right)  \text{, any }(p,q,X)\in\mathcal{P}^{2,+}%
u_{i}(t,x):\smallskip\\
& \quad\quad\quad\quad\quad\quad p+\Phi_{i}\left(  t,x,u\left(  t,x\right)
,q,X\right)  \geq\left(  \varphi_{i}\right)  _{-}^{\prime}\big(u_{i}\left(
t,x\right)  \big),\;\;\text{or}\smallskip\\
& \quad\quad\quad\quad\quad\quad\Gamma_{i}\left(  t,x,u\left(  t,x\right)
,q\right)  \geq\left(  \psi_{i}\right)  _{-}^{\prime}\big(u_{i}\left(
t,x\right)  \big)
\end{array}
\]
together with $u$ is a viscosity super-solution that is, for any
$i\in\overline{1,m}$ :%
\[%
\begin{array}
[c]{ll}%
\left(  c\right)  \quad & \text{for any }\left(  t,x\right)  \in
(0,T)\times\overline{{D}}\text{, any }(p,q,X)\in\mathcal{P}^{2,-}%
u_{i}(t,x):\smallskip\\
& \quad\quad\quad\quad\quad\quad p+\Phi_{i}\left(  t,x,u\left(  t,x\right)
,q,X\right)  \leq\left(  \varphi_{i}\right)  _{+}^{\prime}\big(u_{i}\left(
t,x\right)  \big),\medskip\\
\left(  d\right)  \quad & \text{for any }\left(  t,x\right)  \in(0,T)\times
Bd\left(  \overline{{D}}\right)  \text{, any }(p,q,X)\in\mathcal{P}^{2,-}%
u_{i}(t,x):\smallskip\\
& \quad\quad\quad\quad\quad\quad p+\Phi_{i}\left(  t,x,u\left(  t,x\right)
,q,X\right)  \leq\left(  \varphi_{i}\right)  _{+}^{\prime}\big(u_{i}\left(
t,x\right)  \big),\;\;\text{or}\smallskip\\
& \quad\quad\quad\quad\quad\quad\Gamma_{i}\left(  t,x,u\left(  t,x\right)
,q\right)  \leq\left(  \psi_{i}\right)  _{+}^{\prime}\big(u_{i}\left(
t,x\right)  \big)
\end{array}
\]

\end{theorem}

\begin{theorem}
\label{PR}(Pardoux, R\u{a}\c{s}canu \cite{Pa-Ra/98} : Theorem 4.1) Assume that
$D\ =\mathbb{R}^{d}$ (the system (\ref{pvi1}) is on $\mathbb{R}^{d}$ without
boundary condition and in (\ref{ch5-pvi-2}) and (\ref{ch5-pvi-3}) $A^{t,x}=0,$
$G=0,$ $\psi=0$). Then the continuous function $u:\left[  0,T\right]
\times\mathbb{R}^{d}\rightarrow\mathbb{R}^{m}$ defined by (\ref{ch5-pvi-6}) is
a viscosity solution of the parabolic differential system (\ref{pvi1}-$\left(
a\right)  \&\left(  c\right)  $) on $\mathbb{R}^{d}$ i.e.
\[
\left\vert \text{\ }%
\begin{array}
[c]{l}%
u(T,x)=\kappa\left(  x\right)  ,\;\forall~x\in\mathbb{R}^{d},\medskip\\
u(t,x)\in Dom\left(  \varphi\right)  ,\ \ \forall{(t,x)}\in(0,T)\times
\mathbb{R}^{d},
\end{array}
\right.
\]
\newline and for any $i\in\overline{1,m}$, any $\left(  t,x\right)
\in(0,T)\times\mathbb{R}^{d}$:%
\begin{align*}
p+\Phi_{i}\left(  t,x,u\left(  t,x\right)  ,q,X\right)   &  \geq\left(
\varphi_{i}\right)  _{-}^{\prime}\big(u_{i}\left(  t,x\right)  \big),\quad
\text{for all }(p,q,X)\in\mathcal{P}^{2,+}u_{i}(t,x),\;\text{and}\medskip\\
p+\Phi_{i}\left(  t,x,u\left(  t,x\right)  ,q,X\right)   &  \leq\left(
\varphi_{i}\right)  _{+}^{\prime}\big(u_{i}\left(  t,x\right)  \big),\quad
\text{for all }(p,q,X)\in\mathcal{P}^{2,-}u_{i}(t,x).
\end{align*}

\end{theorem}

We highlight that in \cite{Pa-Ra/98} and \cite{ma-ra/10} the results are given
for $m=1,$ but with the same proof the results hold too for the
quasi-decoupled system (\ref{pvi1}).\medskip

Consider now the parabolic multivalued system (\ref{ch5-pvi-1}) with
${D=}\mathbb{R}^{d}$ and $F$ independent of the last argument $w$ that is
$F\left(  t,x,y,w\right)  \equiv F\left(  t,x,y\right)  \in\mathbb{R}^{m}$ for
all $\left(  t,x,y,w\right)  \in\left[  0,T\right]  \times\mathbb{R}^{d}%
\times\mathbb{R}^{m}\times\mathbb{R}^{m\times m}:$
\begin{equation}
\left\{
\begin{array}
[c]{l}%
-\dfrac{\partial u(t,x)}{\partial t}-\mathcal{L}_{t}u\left(  t,x\right)
+{\partial\varphi}\big(u(t,x)\big)\ni F\big(t,x,u(t,x)\big),\smallskip\\
\multicolumn{1}{r}{t\in\left(  0,T\right)  ,\;x\in\mathbb{R}^{d},\medskip}\\
u(T,x)=\kappa(x),\;\ x\in\mathbb{R}^{d},
\end{array}
\right.  \label{pvi3}%
\end{equation}
Let $z\in\mathbb{R}^{m}$ and $\Phi_{z}:\left[  0,T\right]  \times
\mathbb{R}^{d}\times\mathbb{R}^{m}\times\mathbb{R}^{d}\times\mathbb{S}%
^{d}\rightarrow\mathbb{R}$%
\[
\Phi_{z}\left(  t,x,y,q,X\right)  =\dfrac{1}{2}\mathrm{Tr}\big((gg^{\ast
})(t,x)X\big)+\langle q,f(t,x)\rangle+\left\langle F\left(  t,x,y\right)
,z\right\rangle
\]

\begin{theorem}
\label{MPRZ}(Maticiuc, Pardoux, R\u{a}\c{s}canu, Zalinescu
\cite{Ma/Pa/Ra/Za:09}: Theorem 6, Theorem 14) The continuous function
$u:\left[  0,T\right]  \times\overline{{D}}\rightarrow\mathbb{R}^{m}$ defined
by (\ref{ch5-pvi-6}) is a viscosity solution of the parabolic differential
system (\ref{pvi3}) i.e.
\[
\left\vert \text{\ }%
\begin{array}
[c]{l}%
u(T,x)=\kappa\left(  x\right)  ,\;\forall~x\in\mathbb{R}^{d},\medskip\\
u(t,x)\in Dom\left(  \varphi\right)  ,\ \ \forall{(t,x)}\in(0,T)\times
\mathbb{R}^{d},\medskip
\end{array}
\right.
\]
\newline and
\begin{equation}%
\begin{array}
[c]{c}%
\text{for any }\left(  t,x\right)  \in(0,T)\times\mathbb{R}^{d}\text{, any
}z\in\mathbb{R}^{m}\text{, any }(p,q,X)\in\mathcal{P}^{2,+}\left\langle
u(t,x),z\right\rangle :\medskip\\
p+\Phi_{z}\left(  t,x,u\left(  t,x\right)  ,q,X\right)  \geq\varphi
_{-}^{\prime}\big(u\left(  t,x\right)  ,z\big).
\end{array}
\label{vs-pvi3}%
\end{equation}

\end{theorem}

We remark that

$\left(  r_{1}\right)  $ $\quad$the condition (\ref{vs-pvi3}) is equivalent
to:%
\[%
\begin{array}
[c]{c}%
\text{for any }\left(  t,x\right)  \in(0,T)\times\mathbb{R}^{d}\text{, any
}z\in\mathbb{R}^{m}\text{, any }(p,q,X)\in\mathcal{P}^{2,-}\left\langle
u(t,x),z\right\rangle :\medskip\\
p+\Phi_{z}\left(  t,x,u\left(  t,x\right)  ,q,X\right)  \leq\varphi
_{+}^{\prime}\big(u\left(  t,x\right)  ,z\big).
\end{array}
\]
\newline

$\left(  r_{2}\right)  $ $\quad$in one dimensional case $(m=1)$ condition
(\ref{vs-pvi3}) means the sub-solution for $z>0$ and a super-solution for
$z<0.$\medskip$\smallskip$

We highlight that in supplementary assumptions the uniqueness of the viscosity
solutions holds too in each case presented here above in this subsection.
Moreover the uniqueness of the viscosity solution of the parabolic variational
inequality (\ref{pvi3}) holds in a larger class of functions $u$ (a weaker
inequality (\ref{vs-pvi3})).

\subsection{Elliptic PDEs}

Assume the hypotheses from Sections \ref{s1} and \ref{s2} are satisfied and
moreover $f,g,F$ and $G$ are independent of time argument, $\kappa=0$,
$\varphi=\psi\equiv0,$ $u_{0}=0$ and $F_{i}$ the $i-$th coordinate of $F$,
depends only on the $i-$th row of the matrix $Z$.

If $h:\overline{D}\rightarrow\mathbb{R}$ is a continuous function, then a pair
$(q,X)\in\mathbb{R}^{d}\times\mathbb{S}^{d}$ is a elliptic super-jet to $h$,
at $x\in\overline{D},$ if for all $x^{\prime}\in\overline{D},$%
\[
h(x^{\prime})\leq h(x)+\langle q,x^{\prime}-x\rangle+\dfrac{1}{2}\langle
X(x^{\prime}-x),x^{\prime}-x\rangle+o(|x^{\prime}-x|^{2});
\]
The set of elliptic super-jets at $x$ is denoted by $\mathcal{P}^{2,+}h(x);$
the set of elliptic sub-jets is defined by $\mathcal{P}_{\mathcal{O}}%
^{2,-}h=-\mathcal{P}_{\mathcal{O}}^{2,+}(-h)$.

Consider the semi-linear elliptic partial differential system with
nonlinear Robin boundary condition:%
\begin{equation}
\left\{
\begin{array}
[c]{r}%
-\mathcal{L}u_{i}\left(  x\right)  =F_{i}(x,u\left(  x\right)  ,\left(  \nabla
u_{i}(x)\right)  ^{\ast}g(x)),\quad\ x\in D,\quad i\in\overline{1,m}%
,\medskip\\
\multicolumn{1}{l}{\dfrac{\partial u_{i}}{\partial n}(x)=G_{i}%
(x,u(x)),\ \,x\in Bd\left(  \overline{D}\right)  ,\quad i\in\overline{1,m}.}%
\end{array}
\right.  \label{epds}%
\end{equation}
where%
\[
\mathcal{L}u_{i}\left(  x\right)  =\dfrac{1}{2}%
{\displaystyle\sum\limits_{j,l=1}^{d}}
\left(  gg^{\ast}\right)  _{j,l}(t,x)\dfrac{\partial^{2}u_{i}(x)}{\partial
x_{j}\partial x_{l}}+%
{\displaystyle\sum\limits_{j=1}^{d}}
f_{j}\left(  t,x\right)  \dfrac{\partial u_{i}(x)}{\partial x_{j}}.
\]
Define $\Phi_{i}$ and $\Gamma_{i}$ as in (\ref{figi}).

\begin{proposition}
(E. Pardoux, S. Zhang \cite{pa-zh/98}: Theorem 5.3) The continuous function
$x\longmapsto u\left(  x\right)  :\overline{{D}}\rightarrow\mathbb{R}^{m}$
given by (\ref{el-cont}) is a viscosity solution of the elliptic partial
differential system (\ref{epds}) i.e.:\newline and $u$ is a viscosity
sub-solution that is, for any $i\in\overline{1,m}$ :%
\[%
\begin{array}
[c]{ll}%
\left(  a\right)  \quad & \Phi_{i}\left(  x,u\left(  x\right)  ,q,X\right)
\geq0,\;\;\text{for any }x\in\overline{{D}}\text{, any }(q,X)\in
\mathcal{P}^{2,+}u_{i}(x),\medskip\\
\left(  b\right)  \quad & \Phi_{i}\left(  x,u\left(  x\right)  ,q,X\right)
\vee\Gamma_{i}\left(  x,u\left(  x\right)  ,q\right)  \geq0\smallskip\\
& \quad\quad\quad\quad\quad\quad\text{for any }x\in Bd\left(  \overline{{D}%
}\right)  \text{, any }(q,X)\in\mathcal{P}^{2,+}u_{i}(x),
\end{array}
\]
together with $u$ is a viscosity super-solution that is, for any
$i\in\overline{1,m}$ :%
\[%
\begin{array}
[c]{ll}%
\left(  c\right)  \quad & \Phi_{i}\left(  x,u\left(  x\right)  ,q,X\right)
\leq0,\;\;\text{for any }x\in\overline{{D}}\text{, any }(q,X)\in
\mathcal{P}^{2,-}u_{i}(x),\medskip\\
\left(  d\right)  \quad & \Phi_{i}\left(  x,u\left(  x\right)  ,q,X\right)
\wedge\Gamma_{i}\left(  x,u\left(  x\right)  ,q\right)  \leq0\smallskip\\
& \quad\quad\quad\quad\quad\quad\text{for any }x\in Bd\left(  \overline{{D}%
}\right)  \text{, any }(q,X)\in\mathcal{P}^{2,-}u_{i}(x),
\end{array}
\]

\end{proposition}

\section{Annex}

\subsection{Convex functions}

Let $\varphi:\mathbb{R}^{m}\rightarrow]-\infty,+\infty]$\ be a proper convex
lower semicontinuous function. We denote $\mathrm{Dom}\left(  \varphi\right)
=\left\{  y\in\mathbb{R}^{m}:\varphi\left(  y\right)  <\infty\right\}  ;$
$\varphi$ is a proper function if $\mathrm{Dom}\left(  \varphi\right)
\neq\emptyset.$

The subdifferential (multivalued) operator $\partial\varphi$ is defined by%
\[
\partial\varphi\left(  y\right)  :=\left\{  \hat{y}\in\mathbb{R}%
^{m}:\left\langle \hat{y},v-y\right\rangle +\varphi\left(  y\right)
\leq\varphi\left(  v\right)  ,\;\forall~v\in\mathbb{R}^{m}\right\}  ;
\]
$\partial\varphi:\mathbb{R}^{m}\rightrightarrows\mathbb{R}^{m}$ is a maximal
monotone operator. We have%
\[
\mathrm{Dom}\left(  \partial\varphi\right)  \overset{def}{=}\left\{
y\in\mathbb{R}^{m}:\partial\varphi\left(  y\right)  \neq\emptyset\right\}
\subset\mathrm{Dom}\left(  \varphi\right)  .
\]
Recall that $\overline{\mathrm{Dom}\left(  \partial\varphi\right)  }%
=\overline{\mathrm{Dom}\left(  \varphi\right)  }$ and $int\left(
\mathrm{Dom}\left(  \partial\varphi\right)  \right)  =int\left(
\mathrm{Dom}\left(  \varphi\right)  \right)  .$

For all $y\in\mathrm{Dom}\left(  \varphi\right)  $ and $z\in\mathbb{R}^{m}$ we
have%
\[
\varphi_{-}^{\prime}\left(  y,z\right)  \overset{def}{=}\lim_{t\nearrow
0}\uparrow\frac{\varphi\left(  y+tz\right)  -\varphi\left(  y\right)  }{t}%
\leq\lim_{t\searrow0}\downarrow\frac{\varphi\left(  y+tz\right)
-\varphi\left(  y\right)  }{t}\overset{def}{=}\varphi_{+}^{\prime}\left(
y,z\right)  .
\]
$\varphi_{-}^{\prime}\left(  y,z\right)  =-\varphi_{+}^{\prime}\left(
y,-z\right)  .$ Moreover%
\[%
\begin{array}
[c]{ccc}%
\hat{y}\in\partial\varphi\left(  y\right)   & \;\Longleftrightarrow\; &
\left\langle \hat{y},z\right\rangle \geq\varphi_{-}^{\prime}\left(
y,z\right)  ,\;\forall~z\in\mathbb{R}^{m},\\
& \;\Longleftrightarrow\; & \left\langle \hat{y},z\right\rangle \leq
\varphi_{+}^{\prime}\left(  y,z\right)  ,\;\forall~z\in\mathbb{R}^{m}.
\end{array}
\]
If $m=1$ we write $\varphi_{-}^{\prime}\left(  y\right)  =\varphi_{-}^{\prime
}\left(  y,1\right)  ,$  $\varphi_{+}^{\prime}\left(  y\right)  =\varphi
_{+}^{\prime}\left(  y,1\right)  $ and we have%
\[
\partial\varphi\left(  y\right)  =\left[  \varphi_{-}^{\prime}\left(
y\right)  ,\varphi_{+}^{\prime}\left(  y\right)  \right]  \cap\mathbb{R}.
\]
Let $\varepsilon>0$. The Moreau--Yosida regularization of $\varphi$ is the
function $\varphi_{\varepsilon}:\mathbb{R}^{m}\rightarrow\mathbb{R}$
\[
\varphi_{\varepsilon}\left(  y\right)  \overset{def}{=}\inf\left\{  \frac
{1}{2\varepsilon}\left\vert y-z\right\vert ^{2}+\varphi\left(  z\right)
:z\in\mathbb{R}^{m}\right\}  .
\]
We mention that $\varphi_{\varepsilon}$ is a $C^{1}$ convex function and (see
e.g. Pardoux \& R\u{a}\c{s}canu \cite{pa-ra/14}, Annex B) for all
$x,y\in\mathbb{R}^{m}$%
\begin{equation}%
\begin{array}
[c]{rl}%
(a)\quad & \varphi_{\varepsilon}\left(  x\right)  =\dfrac{\varepsilon}%
{2}\left\vert \nabla\varphi_{\varepsilon}(x)\right\vert ^{2}+\varphi\left(
x-\varepsilon\nabla\varphi_{\varepsilon}(x)\right)  ,\medskip\\
(b)\quad & \nabla\varphi_{\varepsilon}(x)=\partial\varphi_{\varepsilon}\left(
x\right)  \in\partial\varphi\left(  x-\varepsilon\nabla\varphi_{\varepsilon
}(x)\right)  ,\medskip\\
(c)\quad & \left\vert \nabla\varphi_{\varepsilon}(x)-\nabla\varphi
_{\varepsilon}(y)\right\vert \leq\dfrac{1}{\varepsilon}\left\vert
x-y\right\vert .
\end{array}
\label{ineq Yosida}%
\end{equation}

\subsection{A backward stochastic inequality}

From Proposition 6.80 (Annex C) in Pardoux \& R\u{a}\c{s}canu \cite{pa-ra/14}
we have

\begin{lemma}
\label{An_L1}\textit{ Let} $\left(  Y,Z\right)  \in S_{m}^{0}\times
\Lambda_{m\times k}^{0}$ \textit{satisfying}%
\[
Y_{t}=Y_{T}+\int_{t}^{T}d\mathcal{K}_{r}-\int_{t}^{T}Z_{r}dB_{r},\;0\leq t\leq
T,\quad\mathbb{P}-a.s.,
\]
where $\mathcal{K}\in S_{m}^{0}$ and $\mathcal{K}_{\cdot}\left(
\omega\right)  \in BV\left(  \left[  0,T\right]  ;\mathbb{R}^{m}\right)
,\;\mathbb{P}-a.s.\;\omega\in\Omega.\smallskip$\newline\textit{Assume be
given\newline}$%
\begin{array}
[c]{c}%
\mathit{\blacktriangle}\quad
\end{array}
L$ is a non-decreasing stochastic process\textit{, }$L_{0}=0,$ \newline$%
\begin{array}
[c]{c}%
\mathit{\blacktriangle}\quad
\end{array}
R$ is a stochastic process\textit{, }$R_{0}=0$ and $R_{\cdot}\left(
\omega\right)  \in BV\left(  \left[  0,T\right]  ;\mathbb{R}^{m}\right)  $,
$\mathbb{P}-a.s.\;\omega\in\Omega,\newline%
\begin{array}
[c]{c}%
\mathit{\blacktriangle}\quad
\end{array}
V$ \textit{a continuous stochastic process}, $V_{0}=0$, $V_{\cdot}\left(
\omega\right)  \in BV\left(  \left[  0,T\right]  ;\mathbb{R}^{m}\right)
,\mathbb{P}-a.s.\;\omega\in\Omega,$ and%
\[
\mathbb{E~}\left(
{\displaystyle\int_{0}^{T}}
e^{2V_{r}}dR_{r}\right)  ^{-}<\infty
\]
\textit{If} $a<1$ and%
\begin{equation}%
\begin{array}
[c]{rl}%
\left(  i\right)  \quad &
\begin{array}
[c]{l}%
\left\langle Y_{r},d\mathcal{K}_{r}\right\rangle \leq\dfrac{a}{2}\left\vert
Z_{r}\right\vert ^{2}dr+\left(  |Y_{r}|^{2}dV_{r}+|Y_{r}|dL_{r}+dR_{r}\right)
\\
\quad\quad\quad\quad\quad\quad\quad\quad\quad\quad\quad\quad\text{as measures
on }\left[  0,T\right]  ,\medskip
\end{array}
\\
\left(  ii\right)  \quad & \mathbb{E~}\sup\limits_{r\in\left[  \tau
,\sigma\right]  }e^{2V_{r}}\left\vert Y_{r}\right\vert ^{2}<\infty,
\end{array}
\label{AnC-bsde-ip1}%
\end{equation}
\textit{then there exists a positive constant} $C_{a}$ , \textit{depending
only} $a,$ \textit{such that}%
\begin{equation}%
\begin{array}
[c]{l}%
\mathbb{E~}\left(  \sup\limits_{r\in\left[  0,T\right]  }\left\vert e^{V_{r}%
}Y_{r}\right\vert ^{2}\right)  +\mathbb{E~}\left(
{\displaystyle\int_{0}^{T}}
e^{2V_{r}}\left\vert Z_{r}\right\vert ^{2}dr\right)  \medskip\\
\leq C_{a}~\mathbb{E}\left[  \left\vert e^{V_{T}}Y_{T}\right\vert ^{2}+\left(
%
{\displaystyle\int_{0}^{T}}
e^{V_{r}}dL_{s}\right)  ^{2}+%
{\displaystyle\int_{0}^{T}}
e^{2V_{r}}dR_{r}\right]  .
\end{array}
\label{anexC-bsde-b2}%
\end{equation}

\end{lemma}

We remark that the proof of Lemma \ref{An_L1} follows the proof of Proposition
6.80 \cite{pa-ra/14}, with a single small change : in the definition of the
localization stopping time, we delete the term containing $R$, and therefore
we do not need to restrict us to the case where $R$ is  non-decreasing.

\section{Erratum\textit{ }}

In this paper we have corrected the proofs of continuity of the function $\left(
t,x\right)  \mapsto u\left(  t,x\right)  =Y_{t}^{t,x}$ from the papers
\cite{ma-ra/10} (Proposition 13 and Corollary 14) and \cite{pa-zh/98}
(Proposition 4.1 and Theorem 5.1).

\paragraph{Acknowledgement}
The work of A. R. was supported by the grant \textquotedblleft Deterministic and
stochastic systems with state constraints\textquotedblright, code 241/05.10.2011.


\begin{thebibliography}{99}                                                                                               %


\bibitem {am-mr/13}A. Aman, N. Mrhardy, {Obstacle problem for SPDE with
nonlinear Neumann boundary condition via reflected generalized backward doubly
SDEs}, \textit{Statist. Probab. Lett.} {\bf83} (2013), 863--874.\vspace{-0.15cm}

\bibitem {di-ou/10}A. Diakhaby, Y. Ouknine, {Reflected BSDE and Locally
Periodic Homogenization of Semilinear PDEs with Nonlinear Neumann Boundary
Condition}, \textit{Stoch. Anal. Appl.} {\bf 28} (2010), 254--273.\vspace{-0.15cm}

\bibitem {hsu/85}P. Hsu, {Probabilistic approach to the Neumann
problem}, \textit{Comm. Pure Appl. Math.} {\bf38} (1985), 445--472.\vspace{-0.15cm}

\bibitem {hu/93}Y. Hu, {Probabilistic interpretation for a system of
quasilinear elliptic partial differential equations under Neumann boundary
conditions}, \textit{Stochastic Process. Appl.} {\bf48} (1993), 107--121.\vspace{-0.15cm}

\bibitem {le/02}A. Lejay, {BSDE driven by Dirichlet process and
semi--linear Parabolic PDE. Application to Homogenization}, \textit{Stochastic
Process. Appl.} {\bf97} (2002), 1--39.\vspace{-0.15cm}

\bibitem {li-ta/13}Juan Li, Shanjian Tang, {Optimal stochastic control
with recursive cost functional of stochastic differential systems reflected in
a domain, }arXiv:1202.1412v3 [math.PR].\vspace{-0.15cm}

\bibitem {li-sz/84}P.L. Lions, A.S. Sznitman, {Stochastic differential
equations with reflecting boundary conditions}, \textit{Comm. Pure Appl. Math.} {\bf37}
(1984), 511--537.\vspace{-0.15cm}

\bibitem {Ma/Pa/Ra/Za:09}L. Maticiuc; E. Pardoux; A. R\u{a}\c{s}canu; A.
Zalinescu~: {Viscosity solutions for systems of parabolic variational
inequalities}, \textit{Bernoulli} {\bf16} (2010), no. 1, 258-273.

\bibitem {ma-ra/10}L. Maticiuc, A. R\u{a}\c{s}canu, {A stochastic
approach to a multivalued Dirichlet--Neumann problem}, \textit{Stochastic Process.
Appl.} {\bf120} (2010), 777--800.\vspace{-0.15cm}

\bibitem {ma-ra/15}L. Maticiuc, A. R\u{a}\c{s}canu, {On the continuity
of the probabilistic representation of a semilinear Neumann--Dirichlet
problem, }http://arxiv.org/abs/1309.4935 (accepted for publication in
\textit{Stochastic Process. Appl.})

\bibitem {ma-ra/15b}L. Maticiuc, A. R\u{a}\c{s}canu, {Backward
stochastic variational inequalities on random interval, } \textit{Bernoulli} {\bf 21},
no. 2 (2015), 1166-1199.

\bibitem {pa-pe/92}E. Pardoux, S. Peng, {Backward stochastic
differential equations and quasilinear parabolic partial differential
equations}, in \textit{Stochastic partial differential equations and their
applications} (B.L. Rozovskii, R.B. Sowers eds.), LNCIS \textbf{176}, Springer
(1992), 200--217.\vspace{-0.15cm}

\bibitem {Pa-Ra/98}E. Pardoux, A.\ R{\u{a}}{\c{s}}canu~:\ {Backward
stochastic differential equations with subdifferential operator and related
variational inequalities}, \textit{Stochastic Process. Appl.} \textbf{76} (1998),
no.~2, 191--215.

\bibitem {pa-ra/14}E. Pardoux, A. R\u{a}\c{s}canu, \textit{Stochastic
differential equations, Backward SDEs, Partial differential equations},
 Stochastic Modelling and Applied Probability {\bf 69} (2014),
Springer.\vspace{-0.15cm}

\bibitem {pa-zh/98}E. Pardoux, S. Zhang, {Generalized BSDEs and
nonlinear Neumann boundary value problems}, \textit{Probab. Theory Related Fields} {\bf110}
(1998), 535--558.\vspace{-0.15cm}

\bibitem {ra-zh/10}Q. Ran, T. Zhang, {Existence and uniqueness of
bounded weak solutions of a semilinear parabolic PDE}, \textit{J. Theoret. Probab.} {\bf23}
(2010), 951--971.\vspace{-0.15cm}

\bibitem {re-ot/10}Y. Ren, M. El Otmani, {Generalized reflected BSDEs
driven by a L\'{e}vy process and an obstacle problem for PDIEs with a
nonlinear Neumann boundary condition}, \textit{J. Comput. Appl. Math.} {\bf 233} (2010),
2027--2043.\vspace{-0.15cm}

\bibitem {re-xi/06}Y. Ren, N. Xia, {Generalized reflected BSDE and an
obstacle problem for PDEs with a nonlinear Neumann boundary condition}, \textit{Stoch.
Anal. Appl.} {\bf24} (2006), 1013--1033.\vspace{-0.15cm}

\bibitem {ri/09}A. Richou, {Ergodic BSDEs and related PDEs with Neumann
boundary conditions}, \textit{Stochastic Process. Appl.} {\bf119} (2009), 2945--2969.
\end{thebibliography}
\end{document}